\documentclass[12pt, a4paper, tikz]{emilsart}
\title{Enumerating 1324-avoiders with few inversions}
\author{Svante Linusson and Emil Verkama}
\date{}

\affiliationblock{Svante Linusson, Emil Verkama}{Department of Mathematics, KTH Royal Institute of Technology, Stockholm, Sweden}{linusson@kth.se, verkama@kth.se}

\usepackage[backend=biber, doi=false, isbn=false, giveninits=true]{biblatex}
\AtBeginBibliography{\small}

\newcommand{\delete}{\smallsetminus}
\newcommand{\gridcolor}{blue!60!black!25}
\newcommand{\pt}{\small\(\bm\times\)}
\newcommand{\upt}{\(\bm\times\)}

\begin{document}
\maketitle

\begin{center}
  Published in \emph{The Electronic Journal of Combinatorics} \cite{linusson_enumerating_2025}
\end{center}

\begin{abstract}
  \noindent 
  We enumerate the numbers \(\av_n^k(1324)\) of 1324-avoiding \(n\)-permutations with exactly \(k\) inversions for all \(k\) and \(n \geq (k+7)/2\). The result depends on a structural characterization of such permutations in terms of a new notion of almost-decomposability. In particular, our enumeration verifies half of a conjecture of Claesson, Jel\'inek and Steingr\'imsson, according to which \(\av_n^k(1324) \leq \av_{n+1}^k(1324)\) for all \(n\) and \(k\). Proving also the other half would improve the best known upper bound for the exponential growth rate of the number of \(1324\)-avoiders from \(13.5\) to approximately \(13.002\).
\end{abstract}

\section{Introduction}

A permutation \(\pi \in \mathfrak S_n\) \emph{contains} a pattern \(\tau \in \mathfrak S_m\) if there exist indices \(i_1 < \ldots < i_m\) such that \(\pi(i_a) < \pi(i_b)\) if and only if \(\tau(a) < \tau(b)\) for all \(a,b \in [m]\). Otherwise, \(\pi\) \emph{avoids} \(\tau\). An \emph{inversion} in \(\pi\) is a pair of indices \((i,j)\) such that \(i < j\) and \(\pi_i > \pi_j\). We denote by \(\Av_n(\tau)\) the set of all permutations of length \(n\) avoiding \(\tau\), and by \(\Av_n^k(\tau) \subseteq \Av_n(\tau)\) those with exactly \(k\) inversions. Furthermore, we set \(\av_n(\tau) = |{\Av_n(\tau)}|\) and \(\av_n^k(\tau) = |{\Av_n^k(\tau)}|\). Two patterns \(\sigma\) and \(\tau\) are called \emph{Wilf equivalent} if \(\av_n(\sigma) = \av_n(\tau)\) for all \(n\).

\subsection{Avoiding 1324}

It is a well-known that
\begin{equation*}
  \av_n(\tau) = C_n = \frac{1}{n+1} \binom{2n}{n}
\end{equation*}
for all patterns \(\tau\) of length three, but determining \(\av_n(\tau)\) for patterns of length four is much more difficult. The patterns of length four have three distinct Wilf equivalence classes (see \cite{backelin_wilf-equivalence_2007,stankova_forbidden_1994}), usually represented by \(1234\), \(1342\) and \(1324\). Exact formulas for \(\av_n(1234)\) and \(\av_n(1342)\) were found by Gessel in 1990 \cite{gessel_symmetric_1990} and B\'ona in 1997 \cite{bona_exact_1997}, respectively, whereas \(1324\) has resisted all attempts at enumeration. The numbers \(\av_n(1324)\) have been determined computationally up to \(n = 50\) (sequence \href{https://oeis.org/A061552}{A061552} in the OEIS \cite{oeis}), but in general not even the asymptotics are well-understood. For a thorough exposition of these topics, see B\'ona \cite{bona_combinatorics_2022}, Kitaev \cite{kitaev_patterns_2011} or Vatter \cite{vatter_permutation_2015}. The \emph{Stanley--Wilf limit}
\begin{equation*}
  L(\tau) = \lim_{n \to \infty} \av_n(\tau)^{1 / n},
\end{equation*} 
exists for all patterns \(\tau\) due to the Marcus--Tardos theorem \cite{arratia_stanley-wilf_1999,marcus_excluded_2004}, but when \(\tau = 1324\), only loose bounds are known. Table \ref{tab:bounds} shows the timeline of the evolution of these bounds; currently they are \(10.27 < L(1324) < 13.5\) \cite{bevan_structural_2020}. Since \(L(1234) = 9\) and \(L(1342) = 8\), \(1324\) is significantly easier to avoid than the other patterns of length four. Conway, Guttmann and Zinn-Justin have convincingly estimated that \(L(1324) \approx 11.600 \pm 0.003\) \cite{conway_1324-avoiding_2015,conway_1324-avoiding_2018}.

\begin{table}[htb]
  \centering
  \begin{tabular}{lll} \toprule
    & Lower & Upper \\ \midrule
    2004. B\'ona \cite{bona_simple_2004} & & 288 \\
    2005. B\'ona \cite{bona_limit_2005} & 9 & \\
    2006. Albert et al. \cite{albert_stanley-wilf_2006} & 9.47 & \\
    2012. Claesson, Jel\'inek and Steingr\'imsson \cite{claesson_upper_2012} & & 16 \\
    2014. B\'ona \cite{bona_new_2014} & & 13.93 \\
    2015. B\'ona \cite{bona_new_2015} & & 13.74 \\
    2015. Bevan \cite{bevan_permutations_2015} & 9.81 \\
    2020. Bevan et al. \cite{bevan_structural_2020} & 10.27 & 13.5 \\ \bottomrule
  \end{tabular}
  \caption{Best known upper and lower bounds for \(L(1324)\) throughout history.}
  \label{tab:bounds}
\end{table}

One possible avenue towards improvement is suggested by a conjecture of Claesson, Jel\'inek and Steingr\'imsson.

\begin{conjecture}[Conjecture 13 in \cite{claesson_upper_2012}] \label{conj:1324anders}
  For all nonnegative integers \(n\) and \(k\), 
  \begin{equation*}
    \av_n^k(1324) \leq \av_{n+1}^k(1324).
  \end{equation*}
\end{conjecture}

As was demonstrated in \cite{claesson_upper_2012}, the conjecture implies a new upper bound \(L(1324) \leq \exp\big(\pi \sqrt{2/3}\big) < 13.002\), using the fact that \(\av_n^k(1324)\) is constant when the number \(k\) of inversions is fixed and \(n \geq k+2\). Our main result proves half of the conjecture.

\begin{theorem} \label{thm:1324}
  For all nonnegative integers \(k\) and \(n \geq \frac{k+7}{2}\),
  \begin{equation*}
    \av_n^k(1324) = a(k) - 4 a(k-n+1) - 6 \sum_{i=0}^{k-n} a(i),
  \end{equation*}
  where \(a(k) = \sum_{i=0}^k p(i) p(k-i)\) and \(p(k)\) is the number of integer partitions of \(k\). In particular,
  \begin{equation*}
    \av_{n+1}^k(1324) - \av_n^k(1324) = 4a(k-n+1) + 2a(k-n) \geq 0,
  \end{equation*}
  and this difference has a combinatorial interpretation.
\end{theorem}

\begin{remark}
  In the language of generating functions, Theorem \ref{thm:1324} states that
  \begin{equation*}
    \av_n^k(1324) = [x^k] \left(P(x)^2 - \frac{R_n(x)}{1-x}\right)
  \end{equation*}
  whenever \(n \geq \frac{k+7}{2}\), where \(P(x) = \sum_{i \geq 0} p(i) x^i\) is the generating function for the partition numbers and \(R_n(x) = 2(2+x) x^{n-1} P(x)^2\). In particular,
  \begin{equation*}
    \av_{n+1}^k(1324) - \av_n^k(1324) = [x^k] R_n(x).
  \end{equation*}
  This wording is more convenient for the proof, especially in Section \ref{sec:difference}.
\end{remark}

The proof relies on a new notion of \emph{almost decomposable} permutations. The following subsection motivates this idea by defining \emph{decomposable} permutations, and explains the constants \(\av_{k+2}^k(1324)\).

\subsection{Direct sums and decomposability}

For two permutations \(\sigma \in \mathfrak S_n\) and \(\tau \in \mathfrak S_m\), we define the \emph{direct sum} \(\sigma \oplus \tau \in \mathfrak S_{n+m}\) by
\begin{equation*}
  (\sigma \oplus \tau)(i) = \begin{cases}
    \sigma(i) & \text{if } i \leq n, \\
    n + \tau(i-n) & \text{if } i > n.
  \end{cases}
\end{equation*}
For example, \(231 \oplus 21\) is obtained in the following way.
\begin{center}
  \begin{tikzpicture}[scale=0.4]
    \clip (0.5,-1.7) rectangle (6,3.5);
    \draw[step=1, thick, \gridcolor] (0.5,0.5) grid (3.5,3.5);
    \foreach \x\y in {1/2,2/3,3/1}
    {
      \node at (\x,\y-0.004) {\pt};
      \node at (\x,-0.3) {\(\y\)};
    }
    \node at (4.85,2) {\(\oplus\)};
  \end{tikzpicture}
  \begin{tikzpicture}[scale=0.4]
    \clip (0.5,-2.2) rectangle (5,2.5);

    \draw[step=1, thick, \gridcolor] (0.5,0.5) grid (2.5,2.5);
    \foreach \x\y in {1/2,2/1}
    {
      \node at (\x,\y-0.004) {\pt};
      \node at (\x,-0.3) {\(\y\)};
    }
    \node at (3.85,1.5) {\(=\)};
  \end{tikzpicture}
  \begin{tikzpicture}[scale=0.4]
    \clip (0.5,-0.7) rectangle (5.5,5.5);
    \fill[blue!15] (0.5,0.5) rectangle (3.5,3.5);
    \fill[blue!15] (3.5,3.5) rectangle (5.5,5.5);

    \draw[step=1, thick, \gridcolor] (0.5,0.5) grid (5.5,5.5);
    \foreach \x\y in {1/2,2/3,3/1,4/5,5/4}
    {
      \node at (\x,\y-0.004) {\pt};
      \node at (\x,-0.3) {\(\y\)};
    }
  \end{tikzpicture}
\end{center}

If a permutation \(\pi\) is the direct sum of two nonempty permutations, we call \(\pi\) \emph{decomposable}, and otherwise \emph{indecomposable}. Notice that \(\pi\) can be written uniquely as a direct sum
\begin{equation*}
  \pi = \pi^{(1)} \oplus \pi^{(2)} \oplus \ldots \oplus \pi^{(c)},
\end{equation*}
where each \emph{component} \(\pi^{(i)}\) is indecomposable. The formula (see \cite[Lemma 8]{claesson_upper_2012})
\begin{equation*}
  \comp(\pi) + \inv(\pi) \geq |\pi|,
\end{equation*}
where \(\comp(\pi)\), \(\inv(\pi)\) and \(|\pi|\) denote the number of components, the number of inversions and the length of \(\pi\), respectively, indicates that a permutation with few inversions should have many components. In particular, if \(\inv(\pi) \leq |\pi| - 2\), then \(\comp(\pi) \geq |\pi| - \inv(\pi) \geq 2\). It is easy to see that a decomposable permutation \(\pi\) avoids \(1324\) if and only if it is of the form
\begin{equation} \label{eq:1324 form}
  \pi = \pi^{(1)} \oplus 1 \oplus 1 \oplus \ldots \oplus 1 \oplus \pi^{(2)},
\end{equation}
where \(\pi^{(1)}\) avoids \(132\) and \(\pi^{(2)}\) avoids \(213\). The \emph{inversion table} \(b_1 b_2 \ldots b_n\) of a \(132\)-avoider of length \(n\), defined by \(b_i = |\{j > i : \pi_j < \pi_i\}|\), is weakly decreasing and therefore -- with the exclusion of trailing \(0\)'s -- a partition of \(\inv(\pi)\). It follows that
\begin{equation*}
  \av_n^k(132) = \av_n^k(213) = p(k) 
\end{equation*}
for all \(n \geq k+1\). Hence, \eqref{eq:1324 form} gives
\begin{equation*}
  \av_n^k(1324) = a(k) = \sum_{i=0}^k p(i) p(k-i) = [x^k] P(x)^2
\end{equation*}
whenever \(n \geq k + 2\), where \(P(x) = \sum_{i \geq 0} p(i) x^i\) \cite[Proposition 15]{claesson_upper_2012}.

\subsection{Interpreting the main result}

The preceding discussion shows that Conjecture \ref{conj:1324anders} holds trivially (with equality) for all \(n \geq k+2\). Our main result, Theorem \ref{thm:1324}, improves this to \(n \geq \frac{k + 7}{2}\), and therefore proves half of the conjecture along with enumerating the corresponding values of \(\av_n^k(1324)\). The strategy is to find an injection \(\Av_n^k(1324) \to \Av_{n+1}^k(1324)\) and analyze it in order to enumerate the permutations not contained in its image. The injection relies on \emph{almost decomposability}, which is related to normal decomposability, so it is not surprising that the partition numbers show up.

It is useful to keep in mind Table \ref{tab:1324 diagram}, in which the entry on row \(n\) and column \(k\) equals \(\av_n^k(1324)\). Conjecture \ref{conj:1324anders} is equivalent to the statement that each column of the diagram is weakly increasing as \(n\) increases. The blue cells indicate the constant parts of each column; the sequence \(1,2,5,10,20,\ldots\) comes from the generating function \(P(x)^2\). The red cells contain the new numbers enumerated by Theorem \ref{thm:1324}. Specifically, the entry in a blue or red cell on row \(n\) and column \(k\) equals
\begin{equation*}
  a(k) - 4 a(k-n+1) - 6 \sum_{i=0}^{k-n} a(i).
\end{equation*}

\begin{table}[p]
  \centering
  \begin{tblr}{
    hline{2} = {1-14}{black},
    vline{2} = {1-13}{black},
    rows = {mode=math, rowsep=1pt},
    cell{-}{-} = {c},
    cell{3}{2} = {bg=blue!40},
    cell{4}{2-3} = {bg=blue!40},
    cell{5}{2-4} = {bg=blue!40},
    cell{6}{2-5} = {bg=blue!40},
    cell{7}{2-6} = {bg=blue!40},
    cell{8}{2-7} = {bg=blue!40},
    cell{9}{2-8} = {bg=blue!40},
    cell{10}{2-9} = {bg=blue!40},
    cell{11}{2-10} = {bg=blue!40},
    cell{12}{2-11} = {bg=blue!40},
    cell{13}{2-12} = {bg=blue!40},
    cell{7}{7} = {bg=red!40},
    cell{8}{8-9} = {bg=red!40},
    cell{9}{9-11} = {bg=red!40},
    cell{10}{10-13} = {bg=red!40},
    cell{11}{11-14} = {bg=red!40},
    cell{12}{12-14} = {bg=red!40},
    cell{13}{13-14} = {bg=red!40},
  }
    n \backslash k & 0 & 1 & 2 & 3 & 4 & 5 & 6 & 7 & 8 & 9 & 10 & 11 & 12 \\
    1 & 1 \\
    2 & 1 & 1 \\
    3 & 1 & 2 & 2 & 1 \\
    4 & 1 & 2 & 5 & 6 & 5 & 3 & 1 \\
    5 & 1 & 2 & 5 & 10 & 16 & 20 & 20 & 15 & 9 & 4 & 1 \\
    6 & 1 & 2 & 5 & 10 & 20 & 32 & 51 & 67 & 79 & 80 & 68 & 49 & 29 \\
    7 & 1 & 2 & 5 & 10 & 20 & 36 & 61 & 96 & 148 & 208 & 268 & 321 & 351 \\
    8 & 1 & 2 & 5 & 10 & 20 & 36 & 65 & 106 & 171 & 262 & 397 & 568 & 784 \\
    9 & 1 & 2 & 5 & 10 & 20 & 36 & 65 & 110 & 181 & 286 & 443 & 664 & 985 \\
    10 & 1 & 2 & 5 & 10 & 20 & 36 & 65 & 110 & 185 & 296 & 467 & 714 & 1077 \\
    11 & 1 & 2 & 5 & 10 & 20 & 36 & 65 & 110 & 185 & 300 & 477 & 738 & 1127 \\
    12 & 1 & 2 & 5 & 10 & 20 & 36 & 65 & 110 & 185 & 300 & 481 & 748 & 1151
  \end{tblr}
  \caption{The numbers \(\av_n^k(1324)\).}
  \label{tab:1324 diagram}
\end{table}

The differences \(\av_{n+1}^k(1324) - \av_n^k(1324)\) are displayed in Table \ref{tab:1324 difference diagram}. The blue \(0\)'s come from the constant part of each column, and the numbers in the red cells are given by
\begin{equation} \label{eq:1324 difference}
  4a(k-n+1) + 2a(k-n).
\end{equation}
The diagram also shows that \(n \geq \frac{k+7}{2}\) is the best possible bound for our method: if \(n < \frac{k+7}{2}\) (and \(k\) is not too small), then \(\av_{n+1}^k(1324) - \av_n^k(1324)\) no longer equals \eqref{eq:1324 difference}. The data of Table \ref{tab:1324 diagram} was provided to us by Anders Claesson; an expanded version is available at \url{https://akc.is/inv-mono/} along with background on Conjecture \ref{conj:1324anders}.

\begin{table}[p]
  \centering
  \begin{tblr}{
    hline{2} = {1-15}{},
    vline{2} = {1-13}{},
    rows = {mode=math, rowsep=1pt},
    cell{-}{-} = {c},
    cell{3}{2} = {bg=blue!40},
    cell{4}{2-3} = {bg=blue!40},
    cell{5}{2-4} = {bg=blue!40},
    cell{6}{2-5} = {bg=blue!40},
    cell{7}{2-6} = {bg=blue!40},
    cell{8}{2-7} = {bg=blue!40},
    cell{9}{2-8} = {bg=blue!40},
    cell{10}{2-9} = {bg=blue!40},
    cell{11}{2-10} = {bg=blue!40},
    cell{12}{2-11} = {bg=blue!40},
    cell{13}{2-12} = {bg=blue!40},
    cell{7}{7} = {bg=red!40},
    cell{8}{8-9} = {bg=red!40},
    cell{9}{9-11} = {bg=red!40},
    cell{10}{10-13} = {bg=red!40},
    cell{11}{11-15} = {bg=red!40},
    cell{12}{12-15} = {bg=red!40},
    cell{13}{13-15} = {bg=red!40},
  }
    n \backslash k & 0 & 1 & 2 & 3 & 4 & 5 & 6 & 7 & 8 & 9 & 10 & 11 & 12 & 13 \\
    1 & 0 & 1 \\
    2 & 0 & 1 & 2 & 1 \\
    3 & 0 & 0 & 3 & 5 & 5 & 3 & 1 \\
    4 & 0 & 0 & 0 & 4 & 11 & 17 & 19 & 15 & 9 & 4 & 1 \\
    5 & 0 & 0 & 0 & 0 & 4 & 12 & 31 & 52 & 70 & 76 & 67 & 49 & 29 & 14 \\
    6 & 0 & 0 & 0 & 0 & 0 & 4 & 10 & 29 & 69 & 128 & 200 & 272 & 322 & 333 \\
    7 & 0 & 0 & 0 & 0 & 0 & 0 & 4 & 10 & 23 & 54 & 129 & 247 & 433 & 672 \\
    8 & 0 & 0 & 0 & 0 & 0 & 0 & 0 & 4 & 10 & 24 & 46 & 96 & 201 & 397 \\
    9 & 0 & 0 & 0 & 0 & 0 & 0 & 0 & 0 & 4 & 10 & 24 & 50 & 92 & 166 \\
    10 & 0 & 0 & 0 & 0 & 0 & 0 & 0 & 0 & 0 & 4 & 10 & 24 & 50 & 100 \\
    11 & 0 & 0 & 0 & 0 & 0 & 0 & 0 & 0 & 0 & 0 & 4 & 10 & 24 & 50 \\
    12 & 0 & 0 & 0 & 0 & 0 & 0 & 0 & 0 & 0 & 0 & 0 & 4 & 10 & 24
  \end{tblr}
  \caption{The numbers \(\av_{n+1}^k(1324) - \av_n^k(1324)\).}
  \label{tab:1324 difference diagram}
\end{table}

\subsection{Structure of the paper}

This paper is organized as follows. In Section \ref{sec:almost decomposable}, we introduce almost decomposable permutations and prove that all permutations in \(\Av_n^k(1324)\) are either decomposable or almost decomposable whenever \(n \geq \frac{k+7}{2}\). In Section \ref{sec:injection}, we construct an injection \(\Av_n^k(1324) \to \Av_{n+1}^k(1324)\). The enumeration of \(\av_{n+1}^k(1324) - \av_n^k(1324)\) based on the injection is performed in Section \ref{sec:difference}. Finally, Section \ref{sec:discussion} contains a discussion of ideas to extend our method to prove more of Conjecture \ref{conj:1324anders}, reasons we have failed to do so, as well as possible improvements to the upper bound for \(L(1324)\) given by the conjecture.

\section{1324-avoiders with few inversions are almost decomposable} \label{sec:almost decomposable}

We will often utilize the \emph{plots} \(\left\{(i, \pi_i) : i \in [n] \right\}\) (in cartesian coordinates) of permutations \(\pi \in \mathfrak S_n\). Inverting \(\pi\) corresponds with reflecting its plot across the line \(y = x\), and the \emph{reverse-complement} \(\rc(\pi)_i = n + 1 - \pi_{n + 1 - i}\) rotates the plot by \(180\) degrees. Both \(\pi^{-1}\) and \(\rc(\pi)\) preserve \(1324\)-avoidance and the number of inversions of \(\pi\), so these are useful operations for us.

We also use the \emph{Rothe diagram} of \(\pi\), which is obtained from the plot of \(\pi\) by drawing lines to north and east from each point \((i, \pi_i)\), and marking the empty coordinate points -- these points are the inversions of \(\pi\). The following figure shows an example.

\smallskip
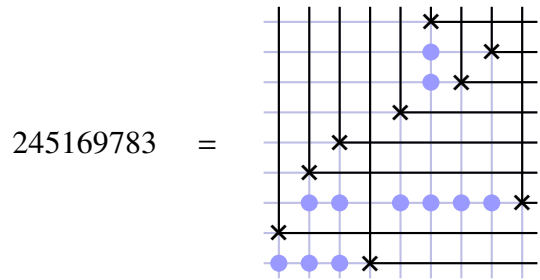
\begin{center}
  \begin{tikzpicture}[scale=0.4]
    \node at (-4.4,5) {\(245169783 \quad =\)};

    \draw[step=1, thick, \gridcolor] (0.5,0.5) grid (9.5,9.5);
    
    \foreach \x\y in {1/2,2/4,3/5,4/1,5/6,6/9,7/7,8/8,9/3}
    {
      \node at (\x,\y-0.004) {\pt};
      \draw[thick]
        (\x,\y) -- (\x,9.5)
        (\x,\y) -- (9.5,\y);
    }
    \foreach \x\y in {1/1,2/1,2/3,3/1,3/3,5/3,6/3,6/7,6/8,7/3,8/3}
      \node[dot, blue!40, minimum size=6pt] at (\x,\y) {};
  \end{tikzpicture}
\end{center}

\begin{definition}
  For \(\pi \in \mathfrak S_n\) and \(i \in [n]\), we denote by \(\pi \delete \pi_i\) the unique permutation in \(\mathfrak S_{n-1}\) that is order-isomorphic to \(\pi_1 \ldots \pi_{i-1} \pi_{i+1} \ldots \pi_n\). We say that \(\pi \delete \pi_i\) is obtained by \emph{deleting} entry \(\pi_i\) from \(\pi\). More generally, if \(S \subseteq [n]\), then \(\pi \delete S\) is the unique permutation that is order-isomorphic to the sequence obtained by removing all entries contained in \(S\) from \(\pi\).
\end{definition}

\begin{definition}\label{def:almost decomp}
  A permutation \(\pi \in \mathfrak S_n\) is called \emph{almost decomposable} if it is indecomposable, but
  \begin{equation*}
    \max\{\comp(\pi \delete i) : i = 1, n, \pi_1, \pi_n\} \geq 2.
  \end{equation*}
\end{definition}

\begin{example}\label{ex:almost decomp}
  Consider the indecomposable permutation \(\pi = 245169783\). Since \(\pi \delete 1 = 13458672 = 1 \oplus 2348671\) is decomposable, \(\pi\) is almost decomposable. However, note that \(\pi \delete 3 = 23415867 = 2341 \oplus 1 \oplus 312\) is also decomposable.

  Almost-decomposability means that deleting one of the points from the `boundary' of the plot of the permutation makes it decomposable. Here are the plots of \(\pi \delete 1\) and \(\pi \delete 3\).

  \begin{center}
    \begin{tikzpicture}[scale=0.4]
      \fill[blue!15] (0.5,1.5) rectangle (1.5,2.5);
      \fill[blue!15] (1.5,2.5) rectangle (9.5,9.5);
    
      \draw[step=1, thick, \gridcolor] (0.5,0.5) grid (9.5,9.5);
    
      \foreach \x\y in {1/2,2/4,3/5,5/6,6/9,7/7,8/8,9/3}
        \node at (\x,\y-0.004) {\pt};

      \node [red] at (4,1-0.004) {\pt};
      \node at (5,-0.5) {\(\pi \delete 1\)};
    \end{tikzpicture} \hspace{5mm}
    \begin{tikzpicture}[scale=0.4]
      \fill[blue!15] (0.5,0.5) rectangle (4.5,5.5);
      \fill[blue!15] (4.5,5.5) rectangle (5.5,6.5);
      \fill[blue!15] (5.5,6.5) rectangle (8.5,9.5);

      \draw [step=1, thick, color=\gridcolor] (0.5,0.5) grid (9.5,9.5);

      \foreach \x\y in {1/2,2/4,3/5,4/1,5/6,6/9,7/7,8/8}
        \node at (\x,\y-0.004) {\pt};

      \node [red] at (9,3-0.004) {\pt};
      \node at (5,-0.5) {\(\pi \delete 3\)};
    \end{tikzpicture}
  \end{center}
\end{example}

An important detail in Section \ref{sec:injection}, for our injection \(\Av_n^k(1324) \to \Av_{n+1}^k(1324)\), is that e.g.\ \emph{both} \(\pi \delete 1\) and \(\pi \delete n\) can be decomposable when \(\pi\) is almost decomposable. However, not all combinations are possible, and this is critical.

\begin{proposition} \label{prop:forbidden removable point combinations}
  Let \(\pi \in \mathfrak S_n\) be indecomposable. If \(\pi \delete 1\) is decomposable then \(\pi \delete \pi_1\) is indecomposable, and similarly if \(\pi \delete n\) is decomposable then \(\pi \delete \pi_n\) is indecomposable.
\end{proposition}

\begin{proof}
  The two parts of the statement are symmetrical, so it suffices to prove the first one. Suppose that \(\pi \delete 1 = \pi^{(1)} \oplus \pi'\), with \(\pi^{(1)}\) an indecomposable permutation. We must have \(\pi^{-1}_1 > |\pi^{(1)}| + 1\) and \(\pi_1 \leq |\pi^{(1)}| + 1\), so that in particular \(\pi_1 < \pi^{-1}_1\). It is therefore not possible that also \(\pi \delete \pi_1\) is decomposable.
\end{proof}

All other combinations are, however, possible. Example \ref{ex:almost decomp} shows an almost decomposable permutation \(\pi\) for which \(\pi \delete 1\) are \(\pi \delete \pi_n\) decomposable, which is one of the four such combinations of two entries.

The goal of this section is to show that, up to the upper bound of \(2n-7\) inversions, every \(1324\)-avoider of length \(n\) is either decomposable or almost decomposable. (We will rewrite the bound \(n \geq \frac{k+7}{2}\) as \(k \leq 2n - 7\) from here on.) This is the structural characterization that our proof of Theorem \ref{thm:1324} relies on.

\begin{theorem} \label{thm:removepoint}
  Each indecomposable permutation \(\pi \in \Av_n^k(1324)\) with \(k \leq 2n - 7\) is almost decomposable.
\end{theorem}

\begin{proof}
  If \(\pi_1 < \pi_n\) and \(\pi^{-1}_1 < \pi^{-1}_n\), then Lemmas \ref{lem:at least one corner} and \ref{lem:corner decomposable} below show that \(\pi\) is almost decomposable. Otherwise Lemma \ref{lem:sides decomposable} applies. 
\end{proof}

All of the facts needed for the above proof are obtained by counting inversions in a specific way, thus showing that all permutations with certain properties violate the bound \(k \leq 2n - 7\). We will need to refer to certain `regions' in the plots of our permutations. To start with, if a permutation \(\pi\) satisfies \(\pi_1 < \pi_n\) and \(\pi^{-1}_1 < \pi^{-1}_n\), then any entry \(\pi_i\) such that
\begin{equation*}
  i < \pi^{-1}_1 \qquad \text{and} \qquad \pi_i > \pi_n
\end{equation*}
is said to \emph{lie in the northwestern region of \(\pi\)}, and if instead
\begin{equation*}
  i > \pi^{-1}_n \qquad \text{and} \qquad \pi_i < \pi_1
\end{equation*}
then \(\pi_i\) \emph{lies in the southeastern region of \(\pi\)}. See Figure \ref{fig:corner regions} for a visualization.

\begin{figure}[htb]
  \begin{minipage}[t]{0.5\textwidth}\vspace{0pt}
    \begin{tikzpicture}[scale=0.55]
      \fill[blue!40] (0,6) rectangle (3,9);
      \fill[red!40] (6,0) rectangle (9,3);
      \draw (0,0) rectangle (9,9);
      \foreach \x\y in {0/3,3/0,9/6,6/9}
      {
        \node at (\x,\y-0.003) {\pt};
        \draw 
          (\x,0) -- (\x,9)
          (0,\y) -- (9,\y);
      }
      \node at (-0.65,2.95) {\(\pi_1\)};
      \node at (3,-0.6) {\(1\)};
      \node at (9.7,5.95) {\(\pi_n\)};
      \node at (6,9.6) {\(n\)};
    \end{tikzpicture}
  \end{minipage}\hfill
  \begin{minipage}[t]{0.5\textwidth}\vspace{5mm}
    \caption{The northwestern and southeastern regions of a permutation \(\pi\), colored blue and red, respectively.}
    \label{fig:corner regions}
  \end{minipage}
\end{figure}

\begin{lemma} \label{lem:at least one corner}
  Suppose \(\pi \in \Av_n(1324)\) satisfies \(\pi_1 < \pi_n\) and \(\pi^{-1}_1 < {\pi^{-1}_n}\). If \(\pi\) is indecomposable, it must have a point in its northwestern or southeastern region.
\end{lemma}

\begin{proof}
  Suppose this is not the case, and let \(m\) be the largest index such that \(\pi_m < \pi_1\). We will show that \(\pi_1 \ldots \pi_m\) is a permutation, and therefore that \(\pi\) is decomposable. Indeed, we must have \(\pi_i < \pi_m\) for all \(\pi_1^{-1} < i < m\), as otherwise \(1 \pi_i \pi_m n\) forms a \(1324\)-pattern. In particular \(\pi_i < \pi_n\) for all \(i \leq m\), and therefore \(\pi_i < \pi_j\) for all \(i \leq m\) and \(j > m\); otherwise \(\pi_1 \pi_i \pi_j \pi_n\) is an occurrence of \(1324\).
\end{proof}

A point in the northwestern or southeastern region intuitively causes many inversions. Indeed, we will show that a \(1324\)-avoider with at most \(2n-6\) inversions can have points in only one of the two regions. 

For the remaining results, it will be convenient to make a distinction between inversions of the form \((j,i)\) and \((i,j)\) for a given index \(i\). The former will be called \emph{left-inversions} of index \(i\), and the latter \emph{right-inversions}. In the Rothe diagram of \(\pi\), left-inversions are located to the left of the point \((i,\pi_i)\), and right-inversions are below it.

\begin{lemma} \label{lem:not both corners}
   If \(\pi \in \Av_n^k(1324)\) with \(k \leq 2n-6\), \(\pi_1 < \pi_n\) and \(\pi^{-1}_1 < {\pi^{-1}_n}\), then either the northwestern or the southeastern region of \(\pi\) contains no points.
\end{lemma}

\begin{proof}
  Suppose \(\pi\) has points \(\pi_i\) and \(\pi_j\) in the northwestern and southeastern regions, respectively. The index \(i\) has  \(\pi_i - i\) right-inversions, since if an index \(k < i\) has \(\pi_k > \pi_i\) then \(\pi_1 \pi_k \pi_i n\) forms a \(1324\)-pattern. Furthermore, indices \(\pi^{-1}_1\) and \(n\) have \(\pi^{-1}_1 - 1\) and \(n - \pi_n\) left-inversions, respectively. Of these inversions, \((i,\pi^{-1}_1)\) and \((i,n)\) were counted twice, so adding them up we get at least
  \begin{equation*}
    \pi_i - i + \pi^{-1}_1 - 1 + n - \pi_n - 2 = n + {\underbrace{\pi^{-1}_1 - i}_{> 0}} + {\underbrace{\pi_i - \pi_n \vphantom{\pi^{-1}_1}}_{> 0}} - 3 \geq n - 1.
  \end{equation*}
  Similarly, counting left-inversions of index \(1\), right-inversions of index \(j\), and left-inversions of index \(\pi^{-1}_n\) gives another \(n-1\), out of which \((1, \pi^{-1}_1)\), \((i,j)\) and \((\pi^{-1}_n, n)\) were counted twice. In total,
  \begin{equation*}
    \inv(\pi) \geq 2(n-1) - 3 = 2n - 5.
  \end{equation*}
  The plot of \(\pi\) is illustrated in Figure \ref{fig:not both corners}. Points whose right-inversions are counted are colored blue, and points whose left-inversions are counted are colored red. The vertical blue rays indicate the possible positions of right-inversions of the corresponding point in the Rothe diagram of \(\pi\); horizontal red rays contain the left-inversions. Intersections of red and blue rays correspond with double-counted inversions. The northwestern and southeastern regions are colored gray.
\end{proof}

\begin{figure}[htb]
  \begin{minipage}[t]{0.5\textwidth}\vspace{0pt}
    \begin{tikzpicture}[scale=0.61]
      \fill[black!15] (1,6.5) rectangle (4,9);
      \fill[black!15] (6.5,1) rectangle (9,3.5);
  
      \draw (1,1) rectangle (9,9);
      
      \draw
        (1,3.5) -- (9,3.5)
        (4,1) -- (4,9);
  
      \foreach \x\y in {1/3.5,3/7.5,6.5/9}
      {
        \draw[line width=4pt, blue!40] 
          (\x,\y-0.2) -- (\x,1);
        \node[blue] at (\x,\y-0.003) {\pt};
      }
      \foreach \x\y in {4/1,7.5/2.25,9/6.5}
      {
        \draw[line width=4pt, red!40] 
          (\x-0.2,\y) -- (1,\y);
        \node[red] at (\x,\y-0.003) {\pt};
      }
      \foreach \x\y in {1/1,1/2.25,3/1,3/2.25,3/6.5,6.5/2.25,6.5/6.5}
        \node[dot, violet!50, minimum size=5.25pt] at (\x,\y) {};
  
      \node at (0.4,3.45) {\(\pi_1\)};
      \node at (2.5,7.9) {\(\pi_i\)};
      \node at (6.5,9.5) {\(n\)};
      \node at (4,0.4) {\(1\)};
      \node at (8.1,1.8) {\(\pi_j\)};
      \node at (9.6,6.45) {\(\pi_n\)};
    \end{tikzpicture}
  \end{minipage}\hfill
  \begin{minipage}[t]{0.5\textwidth}\vspace{5mm}
    \caption{Counting inversions in a permutation with points in both the nortwestern and the southeastern region. Proof of Lemma \ref{lem:not both corners}.}
    \label{fig:not both corners}
  \end{minipage}
\end{figure}

Lemmas \ref{lem:at least one corner} and \ref{lem:not both corners} together prove that if \(\pi \in \Av_n^k(1324)\) is indecomposable, \(k \leq 2n-6\), \(\pi_1 < \pi_n\) and \(\pi^{-1}_1 < \pi^{-1}_n\), then \(\pi\) has a point either in the northwestern or the southeastern region, but not in both. Observe that \(\pi_i\) lies in the northwestern region of \(\pi\) if and only if \(i\) lies in the southeastern region of \(\pi^{-1}\), so it always suffices to examine only one of the two cases.

Our following result will explain why \(k \leq 2n-7\) is the best possible upper bound for our method. We introduce some more terminology based on Figure \ref{fig:side regions}. Suppose \(\pi\) has a point \(\pi_i\) in the northwestern region. A point \(\pi_j\) satisfying
\begin{equation*}
  i < j < \pi^{-1}_n \qquad \text{and} \qquad 1 < \pi_j < \pi_1,
\end{equation*}
is said to \emph{lie in the southern region in relation to \(\pi_i\)}. If
\begin{equation*}
  \pi^{-1}_n < j < n \qquad \text{and} \qquad \pi_1 < \pi_j < \pi_n,
\end{equation*}
then we say that \(\pi_j\) \emph{lies in the eastern region in relation  to \(\pi_i\)}. The points \(1\) and \(\pi_n\) are excluded from these regions. Lastly, notice that if \(\pi_j\) satisfies
\begin{equation*}
    i < j < \pi^{-1}_n \qquad \text{and} \qquad \pi_1 < \pi_j < \pi_n,
\end{equation*}
then \(\pi_1 \pi_i \pi_j n\) forms a \(1324\) pattern. This is why the central yellow region in Figure \ref{fig:side regions} must be empty if \(\pi\) avoids \(1324\).

\begin{figure}[htb]
  \begin{minipage}[t]{0.5\textwidth}\vspace{0pt}
    \begin{tikzpicture}[scale=0.61]
      \fill[blue!40] (3,1) rectangle (7,3);
      \fill[red!40] (7,3) rectangle (9,7);
      \fill[yellow!40] (3,3) rectangle (7,7);
  
      \draw (1,1) rectangle (9,9);
  
      \foreach \x\y in {1/3,3/7,7/9,4.5/1,9/5}
      {
        \draw 
          (\x,1) -- (\x,\y) -- (9,\y);
        \node at (\x,\y-0.003) {\pt};
      }
      \node at (0.4,2.95) {\(\pi_1\)};
      \node at (2.5,7.4) {\(\pi_i\)};
      \node at (7,9.5) {\(n\)};
      \node at (4.5,0.4) {\(1\)};
      \node at (9.6,4.95) {\(\pi_n\)};
    \end{tikzpicture}
  \end{minipage}\hfill
  \begin{minipage}[t]{0.5\textwidth}\vspace{5mm}
    \caption{The southern region (in blue) and eastern region (in red) in relation to the northwestern point \(\pi_i\) of a permutation \(\pi\). The central yellow region is empty if \(\pi\) avoids \(1324\).}
    \label{fig:side regions}
  \end{minipage}
\end{figure}

\begin{lemma} \label{lem:no corner and both sides}
  If \(\pi \in \Av_n^k(1324)\), \(k \leq 2n-7\), with \(\pi_1 < \pi_n\) and \(\pi^{-1}_1 < \pi^{-1}_n\) has a point \(\pi_i\) in the northwestern region, then there cannot exist two points \(\pi_{j_1}\) and \(\pi_{j_2}\) that lie in the southern and eastern regions in relation to \(\pi_i\), respectively.
\end{lemma}

\begin{figure}[htb]
  \begin{minipage}[t]{0.5\textwidth}{\vspace{0pt}}
    \begin{tikzpicture}[scale=0.61]
      \fill[black!15] (3,1) rectangle (7,3);
      \fill[black!15] (7,3) rectangle (9,7);
  
      \draw (1,1) rectangle (9,9);
  
      \foreach \x\y in {1/3,3/7,7/9}
      {
        \draw 
          (\x,\y) -- (9,\y);
        \draw[line width=4pt, blue!40] 
          (\x,\y-0.2) -- (\x,1);
        \node[blue] at (\x,\y-0.003) {\pt};
      }
      \foreach \x\y in {5/1,4/2,8/4.2,9/6}
      {
        \draw[line width=4pt, red!40] 
          (\x-0.2,\y) -- (1,\y);
        \node[red] at (\x,\y-0.003) {\pt};
      }
      \foreach \x\y in {1/1,1/2,3/1,3/2,3/4.2,3/6,7/4.2,7/6}
        \node[dot, violet!50, minimum size=5.25pt] at (\x,\y) {};
  
      \node at (8,2) {\(\emptyset\)};
      \node at (5,5) {\(\emptyset\)};
      \node at (0.4,2.95) {\(\pi_1\)};
      \node at (2.5,7.4) {\(\pi_i\)};
      \node at (7,9.5) {\(n\)};
      \node at (5,0.4) {\(1\)};
      \node at (4.7,1.9) {\(\pi_{j_1}\)};
      \node at (8,4.7) {\(\pi_{j_2}\)};
      \node at (9.6,5.95) {\(\pi_n\)};
    \end{tikzpicture}
  \end{minipage}\hfill
  \begin{minipage}[t]{0.5\textwidth}{\vspace{5mm}}
    \caption{Counting inversions of a permutation with points in both the southern and eastern regions in relation to a point \(\pi_i\) in the northwestern region. Proof of Lemma \ref{lem:no corner and both sides}.}
    \label{fig:no corner and both sides}
  \end{minipage}
\end{figure}

\begin{proof}
  Suppose that such indices \(j_1\) and \(j_2\) exist. See Figure \ref{fig:no corner and both sides} for the plot of \(\pi\). There can be no point in the southeastern region of \(\pi\) by Lemma \ref{lem:not both corners}, and, as discussed above, the `central' region in relation to \(\pi_i\) must be empty. It follows that each right-inversion of index \(i\) -- except for \((i, \pi^{-1}_1)\) and \((i, n)\) -- is caused by a point in the southern or eastern regions in relation to \(\pi_i\). By counting the maximal number of points that can lie in these regions (and including \(1\) and \(\pi_n\)), we find that the number \(\pi_i - i\) of right-inversions of index \(i\) is at most \(\pi_1 - 1 + n - \pi^{-1}_n\).

  Furthermore, index \(j_1\) has at least \(i\) left-inversions, since if \(j < i\) and \(\pi_j < \pi_{j_1}\), then \(\pi_j \pi_i \pi_{j_1} n\) forms a \(1324\)-pattern. With a symmetrical argument, \(j_2\) has at least \(n - \pi_i + 1\) left-inversions. Adding together the right-inversions of indices \(i\), \(1\) and \(\pi^{-1}_n\) with the left-inversions of indices \(\pi^{-1}_1\), \(j_1\), \(j_2\) and \(n\), accounting for the double counting of \((1,\pi^{-1}_1)\), \((1,j_1)\), \((i,\pi^{-1}_1)\), \((i,j_1)\), \((j_1,j_2)\), \((i,n)\), \((\pi^{-1}_n,j_1)\) and \((\pi^{-1}_n,n)\), we finally get
  \begin{align*}
    \inv(\pi) &\geq 
      \pi_i - i \ \ +\ \ 
      \underbrace{\pi_1 - 1 \ \ +\ \   
      n - \pi^{-1}_n}_{\geq \pi_i - i} \\
    &\quad +\ \   
      {\underbrace{\pi^{-1}_1 - 1}_{\geq i}} \ \ +\ \  
      i \ \ +\ \  
      n - \pi_i + 1 \ \ +\ \  
      {\underbrace{n - \pi_n \vphantom{\pi^{-1}_1}}_{\geq n - \pi_i + 1}}
      \ \ -\ \  8 \\
    &\geq 2(n + i - \pi_i + 1) + 2(\pi_i - i) - 8 \\
    &= 2n - 6. \tag*{\qedhere}
  \end{align*}
\end{proof}

The bound \(k \leq 2n - 7\) is tight, since e.g.\ 
\begin{equation*}
  \pi = 3612\ \ 7\ldots n\ \ 45
\end{equation*}
is \(1324\)-avoiding with \(2n-6\) inversions, and neither decomposable nor almost decomposable. Here is the plot of the first such permutation, \(\pi = 3612745\).

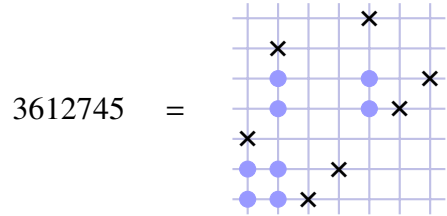
\begin{center} \noindent
  \begin{tikzpicture}[scale=0.4]
    \node at (-3.9,4) {\(3612745 \quad = \)};
    \draw [step=1, thick, \gridcolor] (0.5,0.5) grid (7.5,7.5);
    \foreach \x\y in {1/3,2/6,3/1,4/2,5/7,6/4,7/5}
      \node at (\x,\y-0.004) {\pt};
    \foreach \x\y in {1/1,1/2,2/1,2/2,2/4,2/5,5/4,5/5}
      \node[dot, blue!40, minimum size=6pt] at (\x,\y) {};
  \end{tikzpicture}
\end{center}

\vspace{-4pt}
\noindent
Interestingly, the same permutation \(\pi\) has been used before to exemplify that two \(1324\)-avoiding permutations can have the same \emph{profile}: \(\pi\) and \(\pi^{-1}\) have the same left-to-right minima and right-to-left maxima in the same positions \cite{bona_permutations_1997}.

Lemma \ref{lem:no corner and both sides} is sufficient for permutations with only one point in the northwestern region, but not strong enough to give almost-decomposability for permutations with several points in the northwestern region. However, as the following result shows, such a condition is even more limiting.

\begin{lemma} \label{lem:no doublecorner and both sides}
  Let \(\pi \in \Av_n^k(1324)\) with \(k \leq 2n - 7\), \(\pi_1 < \pi_n\), \(\pi^{-1}_1 < \pi^{-1}_n\), and two points \(\pi_{i_1} < \pi_{i_2}\) in the northwestern region. There cannot exist two points \(\pi_{j_1}\) and \(\pi_{j_2}\) that lie in the southern region in relation to \(\pi_{i_1}\) and the eastern region in relation to \(\pi_{i_2}\), respectively.
\end{lemma}

\begin{proof}
  Lemma \ref{lem:no corner and both sides} shows that there are no points in the eastern region in relation to \(\pi_{i_1}\) or the southern region in relation to \(\pi_{i_2}\). It follows that the points causing right-inversions for indices \(i_1\) and \(i_2\) (excluding \(1\) and \(\pi_n\)) are contained in the southern and eastern gray regions of Figure \ref{fig:no doublecorner and both sides}, respectively, and none of the points creates a right-inversion for \emph{both} \(i_1\) and \(i_2\). On the other hand, the points \(1\) and \(\pi_n\) create right-inversions for both \(i_1\) and \(i_2\), so the total number \(\pi_{i_1} - i_1 + \pi_{i_2} - i_2\) of right-inversions of \(i_1\) and \(i_2\) satisfies
  \begin{equation*}
    \pi_{i_1} - i_1 + \pi_{i_2} - i_2 \leq \pi_1 - 1 + n - \pi^{-1}_n + 2.
  \end{equation*}
  Counting inversions of \(\pi\) as in the previous proof, we get
  \begin{align*}
    \inv(\pi) &\geq
      \pi_{i_1} - i_1
      \ \ +\ \ \pi_{i_2} - i_2
      \ \ +\ \ {\underbrace{\pi_1 - 1
      \ \ +\ \ n - \pi^{-1}_n}_{\geq \pi_{i_1} - i_1 + \pi_{i_2} - i_2 - 2}} \\
      &\quad +\ \ {\underbrace{\pi^{-1}_1 - 1}_{\geq i_2}}
      \ \ +\ \ i_1
      \ \ +\ \ n - \pi_{i_2} + 1
      \ \ +\ \ \underbrace{n - \pi_n}_{\geq n - \pi_{i_1} + 1}
      \ \ -\ \ 10 \\
    &\geq 2(\pi_{i_1} - i_1 + \pi_{i_2} - i_2) + i_2 + i_1 + n - \pi_{i_2} + n - \pi_{i_1} - 10 \\
    &= 2n + {\underbrace{\pi_{i_1} - i_1}_{\geq 3}} + {\underbrace{\pi_{i_2} - i_2}_{\geq 3}} - 10 \\
    &\geq 2n - 4.
  \end{align*}
  Figure \ref{fig:no doublecorner and both sides} shows the inversions we counted.
\end{proof}

\begin{figure}[htb]
  \begin{minipage}[t]{0.5\textwidth}\vspace{0pt}
    \begin{tikzpicture}[scale=0.61]
      \fill[black!15] (2.5,1) rectangle (3.5,3);
      \fill[black!15] (7,6.5) rectangle (9,7.5);
  
      \draw (1,1) rectangle (9,9);
  
      \foreach \x\y in {1/3,2.5/6.5,3.5/7.5,7/9}
      {
        \draw 
          (\x,\y) -- (9,\y);
        \draw[line width=4pt, blue!40] 
          (\x,\y-0.2) -- (\x,1);
        \node[blue] at (\x,\y-0.003) {\pt};
      }
      \foreach \x\y in {5/1,3/2,8/7,9/5.5}
      {
        \draw[line width=4pt, red!40] 
          (\x-0.2,\y) -- (1,\y);
        \node[red] at (\x,\y-0.003) {\pt};
      }
      \foreach \x\y in {1/1,1/2,2.5/1,2.5/2,2.5/5.5,3.5/1,3.5/5.5,3.5/7,7/5.5,7/7}
        \node[dot, violet!50, minimum size=5.25pt] at (\x,\y) {};
  
      \node at (8,2) {\(\emptyset\)};
      \node at (5.25,4.75) {\(\emptyset\)};
      \node at (0.4,3) {\(\pi_1\)};
      \node at (1.9,6.5) {\(\pi_{i_1}\)};
      \node at (3,7.9) {\(\pi_{i_2}\)};
      \node at (7,9.5) {\(n\)};
      \node at (5,0.4) {\(1\)};
      \node at (3.05,2.5) {\(\pi_{j_1}\)};
      \node at (8.6,6.9) {\(\pi_{j_2}\)};
      \node at (9.6,5.45) {\(\pi_n\)};
    \end{tikzpicture}
  \end{minipage}\hfill
  \begin{minipage}[t]{0.5\textwidth}\vspace{5mm}
    \caption{Counting inversions of a permutation with points in the southern and eastern regions in relation to the points \(\pi_{i_1}\) and \(\pi_{i_2}\) in the northwestern region. Proof of Lemma \ref{lem:no doublecorner and both sides}.}
    \label{fig:no doublecorner and both sides}
  \end{minipage}
\end{figure}

We are now prepared to prove one of our central lemmas.

\begin{lemma} \label{lem:corner decomposable}
  If \(\pi \in \Av_n^k(1324)\) with \(k \leq 2n-7\) and \(\pi_1 < \pi_n\), \(\pi^{-1}_1 < \pi^{-1}_n\) has a point in its northwestern region, then \(\pi \delete 1\) or \(\pi \delete \pi_n\) is decomposable.
\end{lemma}

\begin{proof}
  Let \(i_1 < i_2 < \ldots < i_m\) be the set of all indices whose points lie in the northwestern region of \(\pi\). First, assume that no point lies in the southern region in relation to \(\pi_{i_1}\) and denote \(\sigma = \pi \delete 1\). We claim that \(\sigma_1 \ldots \sigma_{i_1 - 1}\) is a permutation, which in turn would imply that \(\sigma\) is decomposable. Indeed, otherwise there exists indices \(j_1\) and \(j_2\) such that
  \begin{equation*}
    j_1 < i_1 < j_2 \qquad \text{and} \qquad \pi_{j_1} > \pi_{j_2}.
  \end{equation*}
  However, we assumed that the southern region in relation to \(i_1\) is empty, so \(\pi_{j_2} > \pi_1\) and therefore \(\pi_1 \pi_{j_1} \pi_{j_2} \pi_n\) forms a \(1324\) pattern. Figure \ref{fig:corner decomposable} visualizes the argument: the red points \(\pi_{j_1}\) and \(\pi_{j_2}\) are in impossible positions, implying that the red region in the bottom right must be empty and therefore that the blue region to its left is a permutation in \(\sigma\).

  If there is a point in the southern region in relation to \(\pi_{i_1}\), then Lemmas \ref{lem:no corner and both sides} and \ref{lem:no doublecorner and both sides} prove that the eastern region in relation to \(\pi_{i_m}\) is empty instead. This case is symmetrical to the first one after reverse-complementation and inversion; more precisely, \(\rc(\pi)^{-1}\) gets us back to the first case and therefore
  \begin{equation*}
    \pi \delete \pi_n = \rc\big(\rc(\pi)^{-1} \delete 1\big)^{-1}
  \end{equation*}
  is decomposable.
\end{proof}

\begin{figure}[htb]
  \begin{minipage}[t]{0.5\textwidth}\vspace{0pt}
    \begin{tikzpicture}[scale=0.61]
      \draw[fill=blue!40] (1,1) rectangle (3,3);
      \draw[fill=red!40] (3,1) rectangle (9,3);
  
      \draw (1,1) rectangle (9,9);
  
      \foreach \x\y in {1/2.5,3/6,4/7,7/9,5/1,9/5}
      {
        \draw 
          (\x,1) -- (\x,\y) -- (9,\y);
        \node at (\x,\y-0.003) {\pt};
      }
  
      \node[red] at (2,4.5) {\pt};
      \node[red] at (6,3.75) {\pt};
      \draw 
        (1,5) -- (9,5);
      
      \foreach \d in {0,0.2,0.4}
        \node[dot, minimum size=1.5pt] at (3.3+\d,6.3+\d) {};
  
      \node at (6,2) {\(\emptyset\)};
      \node at (0.4,2.45) {\(\pi_1\)};
      \node at (2.5,6.5) {\(\pi_{i_1}\)};
      \node at (3.5,7.5) {\(\pi_{i_m}\)};
      \node at (7,9.5) {\(n\)};
      \node at (5,0.4) {\(1\)};
      \node at (9.6,4.95) {\(\pi_n\)};
      \node at (2,3.9) {\(\pi_{j_1}\)};
      \node at (6,4.3) {\(\pi_{j_2}\)};
    \end{tikzpicture}
  \end{minipage}\hfill
  \begin{minipage}[t]{0.5\textwidth}\vspace{5mm}
    \caption{A permutation \(\pi\) satisfying the assumptions of Lemma \ref{lem:corner decomposable} is almost decomposable, since \(\pi \delete 1\) or \(\pi \delete \pi_n\) is decomposable. The points \(\pi_{j_1}\) and \(\pi_{j_2}\) cannot be placed as they are in the picture, since they create a \(1324\) pattern.}
    \label{fig:corner decomposable}
  \end{minipage}
\end{figure}

So far, we have always assumed that \(\pi_1 < \pi_n\) and \(\pi^{-1}_1 < \pi^{-1}_n\). We will now assume the opposite, leading to the other half of Theorem \ref{thm:removepoint}. This turns out to be much easier, since if e.g.\ \(\pi_1 > \pi_n\), then deleting \(\pi_1\) and \(\pi_n\) gets rid of a lot of inversions.

\begin{lemma} \label{lem:sides decomposable}
  If \(\pi \in \Av_n^k(1324)\), \(k\leq 2n - 5\) and \(\pi_1 > \pi_n\), then \(\pi \delete \pi_1\) or \(\pi \delete \pi_n\) is decomposable.
\end{lemma}

\begin{proof}
  Let \(\delta = \pi_1 - \pi_n\) and \(\sigma = \pi \delete \{\pi_1, \pi_n\}\). The right-inversions of index \(1\) and left-inversions of index \(n\) in \(\pi\) sum up to
  \begin{equation*}
    \pi_1 - 1 + n - \pi_n - 1 = n - 2 + \delta,
  \end{equation*}
  so
  \begin{equation*}
    \comp(\sigma) \geq |\sigma| - \inv(\sigma) \geq n - 2 - (2n - 5 - n + 2 - \delta) = \delta + 1 \geq 2.
  \end{equation*}
  Write \(\sigma = \sigma^{(1)} \oplus 1 \oplus \ldots \oplus 1 \oplus \sigma^{(2)}\). We must have \(\pi_1 < n - |\sigma^{(2)}| + 1\) or \(\pi_n > |\sigma^{(1)}|\), since otherwise
  \begin{equation*}
    \pi_1 - \pi_n \geq n - |\sigma^{(2)}| + 1 - |\sigma^{(1)}| = \delta + 2.
  \end{equation*}
  If \(\pi_1 < n - |\sigma^{(2)}| + 1\) then \(\pi \delete \pi_n\) is decomposable, and if  \(\pi_n > |\sigma^{(1)}|\) then \(\pi \delete \pi_1\) is decomposable.
\end{proof}

\section{The injection} \label{sec:injection}

Denote by \(\mathcal D_n^k\) and \(\mathcal A_n^k\) the sets of decomposable and almost decomposable permutations in \(\Av_n^k(1324)\), respectively. In this section we will construct injections
\begin{equation*}
  g : \mathcal D_n^k \longrightarrow \Av_{n+1}^k(1324) \qquad \text{and} \qquad
  f : \mathcal A_n^k \longrightarrow \Av_{n+1}^k(1324),
\end{equation*}
with disjoint images, for all \(n\) and \(k\). If \(k \leq 2n - 7\) then all permutations in \(\Av_n^k(1324)\) are decomposable or almost decomposable by Theorem \ref{thm:removepoint}, so our mappings combine to an injection
\begin{equation*}
  \Av_n^k(1324) \longrightarrow \Av_{n+1}^k(1324).
\end{equation*}
In particular, this verifies Conjecture \ref{conj:1324anders} for all \(k \leq 2n-7\). 

First of all, any \(\pi \in \mathcal D_n^k\) can be written in the form
\begin{equation*}
  \pi = \pi^{(1)} \oplus {\underbrace{1 \oplus \ldots \oplus 1}_{m \text{ times}}} \oplus \pi^{(2)}
\end{equation*}
for some \(m \geq 0\) by \eqref{eq:1324 form}. This allows us to set
\begin{equation*}
  g(\pi) = \pi^{(1)} \oplus {\underbrace{1 \oplus \ldots \oplus 1}_{m+1 \text{ times}}} \oplus \pi^{(2)} \in \mathcal D_{n+1}^k.
\end{equation*}
The image \(g(\mathcal D_n^k)\) is exactly the set of all permutations in \(\Av_{n+1}^k(1324)\) with at least three components, and \(g\) is clearly injective. Note that when \(n \geq k + 2\), \(g\) is a bijection.

We will define \(f\) in a similar way. Let \(\pi \in \mathcal A_n^k\).
\begin{enumerate}
  \item If \(\pi \delete \pi_1\) is decomposable, let \(f(\pi)\) be the permutation with \(f(\pi)_1 = \pi_1\) and \(f(\pi) \delete \pi_1 = g(\pi \delete \pi_1)\). \label{inj:delete pi1}
  \item If \(\pi \delete 1\) is decomposable, let \(f(\pi) = f(\pi^{-1})^{-1}\). \label{inj:delete 1}
  \item Otherwise, let \(f(\pi) = ({\rc} \circ f \circ {\rc})(\pi)\), where \(\rc(\pi)\) is the reverse-complement. \label{inj:delete other}
\end{enumerate}

\begin{remark} \label{rmk:f properties}
  Let \(\pi \in \mathcal A_n^k\).
  \begin{enumerate}[(a)] 
    \item It is impossible for both \(\pi \delete 1\) and \(\pi \delete \pi_1\) to be decomposable by Proposition \ref{prop:forbidden removable point combinations}, so parts \ref{inj:delete pi1} and \ref{inj:delete 1} of the definition are exclusive.
    \item If \(\pi \delete 1\) and \(\pi \delete \pi_1\) are indecomposable, then \(\pi \delete n\) or \(\pi \delete \pi_n\) must be decomposable, and it follows that \(\rc(\pi) \delete 1\) or \(\rc(\pi) \delete \rc(\pi)_1\) is decomposable. This is why \((f \circ {\rc})(\pi)\) exists when it is used in part \ref{inj:delete other}.
    \item If \(\pi \delete \pi_1\) is decomposable then \(\pi_1 > \pi_2\), and therefore \(f(\pi)\) avoids \(1324\). The \(1324\)-avoiders are closed under inversion and taking reverse-complements, so \(f(\pi)\) avoids \(1324\) also in parts \ref{inj:delete 1} and \ref{inj:delete other}.
    \item The number of inversions is preserved: in part \ref{inj:delete pi1},
    \begin{equation*}
      \inv(f(\pi)) = \inv(g(\pi \delete \pi_1)) + \pi_1 - 1 = \inv(\pi \delete \pi_1) + \pi_1 - 1 = \inv(\pi).
    \end{equation*}
    Taking the inverse or the reverse-complement preserves the number of inversions, so this is true also for parts \ref{inj:delete 1} and \ref{inj:delete other}.
    \item \(f(\pi)\) has at most two components: in part \ref{inj:delete pi1}, if \(f(\pi) = \pi^{(1)} \oplus 1 \oplus \pi'\) then \(\pi = \pi^{(1)} \oplus \pi'\), a contradiction. Taking the inverse or the reverse-complement preserves the number of components, so this is true also for parts \ref{inj:delete 1} and \ref{inj:delete other}.\label{rmk:f properties: two comps}
    \item In part \ref{inj:delete other}, if \(\pi \delete \pi_n\) is decomposable then
    \begin{align*}
      f(\pi)_{n+1} = ({\rc} \circ f \circ {\rc})(\pi)_{n+1} &= n + 2 - (f \circ {\rc})(\pi)_{n + 2 - n - 1} \\
      &= n + 2 - \rc(\pi)_1 \\
      &= n + 2 - (n + 1 - \pi_{n + 1 - 1}) = \pi_n + 1.
    \end{align*}
    In part \ref{inj:delete 1} it is clear that \(f(\pi)^{-1}_1  = \pi^{-1}_1\), so if instead \(\pi \delete n\) is decomposable in part \ref{inj:delete other} then similarly \(f(\pi)^{-1}_{n+1} = \pi^{-1}_n + 1\). \label{rmk:f properties: edge values}
  \end{enumerate}
\end{remark}

\begin{example}
  Consider the permutation \(\pi = 35126874 \in \Av_8^8(1324)\). Here are the plots of \(\pi\) and \(\rc(\pi)\).
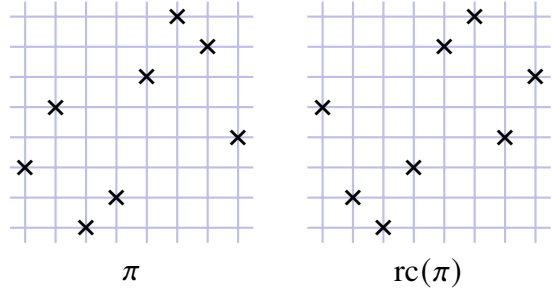
  \begin{center}
    \begin{tikzpicture}[scale=0.4]
      \clip (0.5,-1) rectangle (8.5,8.5);
      \draw[step=1, thick, \gridcolor] (0.5,0.5) grid (8.5,8.5);

      \foreach \x\y in {1/3,2/5,3/1,4/2,5/6,6/8,7/7,8/4}
        \node at (\x,\y-0.004) {\pt};
      
      \node at (4.5,-0.5) {\(\pi\)};
    \end{tikzpicture}
    \hspace{5mm}
    \begin{tikzpicture}[scale=0.4]
      \clip (0.5,-1) rectangle (8.5,8.5);
      \draw[step=1, thick, \gridcolor] (0.5,0.5) grid (8.5,8.5);

      \foreach \x\y in {1/5,2/2,3/1,4/3,5/7,6/8,7/4,8/6}
        \node at (\x,\y-0.004) {\pt}; 
      
      \node at (4.5,-0.5) {\(\rc(\pi)\)};
    \end{tikzpicture}
  \end{center}
  We can see that \(\pi \delete 1\) and \(\pi \delete \pi_1\) are both indecomposable, whereas \(\pi \delete \pi_n\) is decomposable. Therefore \(\rc(\pi) \delete \rc(\pi)_1\) is decomposable. The following figure shows the permutations \(\rc(\pi) \delete \rc(\pi)_1\), \(g(\rc(\pi) \delete \rc(\pi_1))\) and \((f \circ {\rc})(\pi)\). 
  \begin{center}
    \begin{tikzpicture}[scale=0.4]
      \fill[blue!15] (0.5,0.5) rectangle (3.5,3.5);
      \fill[blue!15] (3.5,3.5) rectangle (7.5,7.5);

      \draw[step=1, thick, \gridcolor] (0.5,0.5) grid (7.5,7.5);

      \foreach \x\y in {1/2,2/1,3/3,4/6,5/7,6/4,7/5}
        \node at (\x,\y-0.004) {\pt}; 
      
      \node at (4,-1.5) {\(\rc(\pi) \delete \rc(\pi)_1\)};
    \end{tikzpicture}
    \hspace{5mm}
    \begin{tikzpicture}[scale=0.4]
      \fill[blue!15] (0.5,0.5) rectangle (3.5,3.5);
      \fill[blue!15] (3.5,3.5) rectangle (4.5,4.5);
      \fill[blue!15] (4.5,4.5) rectangle (8.5,8.5);

      \draw[step=1, thick, \gridcolor] (0.5,0.5) grid (8.5,8.5);

      \foreach \x\y in {1/2,2/1,3/3,4/4,5/7,6/8,7/5,8/6}
        \node at (\x,\y-0.004) {\pt}; 

      \node at (4.5,-1) {\(g(\rc(\pi) \delete \rc(\pi)_1)\)};
    \end{tikzpicture}
    \hspace{5mm}
    \begin{tikzpicture}[scale=0.4]
      \fill[blue!15] (1.5,0.5) rectangle (4.5,3.5);
      \fill[blue!15] (4.5,3.5) rectangle (5.5,4.5);
      \fill[blue!15] (5.5,5.5) rectangle (9.5,9.5);

      \draw[step=1, thick, \gridcolor] (0.5,0.5) grid (9.5,9.5);

      \foreach \x\y in {1/5,2/2,3/1,4/3,5/4,6/8,7/9,8/6,9/7}
        \node at (\x,\y-0.004) {\pt}; 

      \node at (5,-0.5) {\((f \circ {\rc})(\pi)\)};
    \end{tikzpicture}
  \end{center}
  Finally, we get the following.
  \begin{center}
    \begin{tikzpicture}[scale=0.4]
      \node at (-10,5) {\(f(\pi) = ({\rc} \circ f \circ {\rc})(\pi) = 341267985 \quad = \)};

      \draw[step=1, thick, \gridcolor] (0.5,0.5) grid (9.5,9.5);

      \foreach \x\y in {1/3,2/4,3/1,4/2,5/6,6/7,7/9,8/8,9/5}
        \node at (\x,\y-0.004) {\pt};
    \end{tikzpicture}
  \end{center}
\end{example}

In order to show that \(f\) is injective, we will construct its inverse. The natural way to obtain \(\pi\) from \(f(\pi) \eqqcolon \sigma\) is by reversing the steps used to construct \(\sigma\): first delete a well-chosen entry in \(\{1,n+1,\sigma_1,\sigma_{n+1}\}\) from \(\sigma\), then apply \(g^{-1}\) on the resulting permutation (which has at least three components), and finally insert back the deleted entry. The choice of entry to delete should mirror the definition of \(f\): deleting \(1\) or \(\sigma_1\) is always prioritized over deleting \(n\) or \(\sigma_n\).
  
However, it is not obvious that this works. Suppose, for example, that \(\pi \delete 1\) and \(\pi \delete \pi_1\) are indecomposable, and \(\pi \delete n\) is decomposable. This means that \(\sigma \delete (n+1)\) has at least three components, and deleting \(n+1\) is indeed correct. But is it possible that \(\sigma \delete 1\) or \(\sigma \delete \sigma_1\) also have at least three components? If, say, \(\sigma \delete 1\) had at least three components, we would prioritize deleting \(1\) over deleting \(n+1\), and (after applying \(g^{-1}\) and inserting back the entry \(1\)) obtain some permutation \(\tau \in \mathcal A_n^k\) such that \(\tau \delete 1\) is decomposable. In particular \(\tau \neq \pi\), so our supposed inverse of \(f\) would not work.

To show that \(f\) is injective, it is therefore crucial to show that the situation we described never happens. In other words, if \(\pi \delete \pi_1\) is indecomposable, then \(f(\pi) \delete f(\pi)_1\) must have at most two components. We prove this in Lemma \ref{lem:not unremovable to removable}, after a simple intermediate result.

\begin{lemma} \label{lem:f fixed points}
  Let \(\pi \in \mathcal A_n^k\) and \(i \in [n]\).
  \begin{enumerate}[(a)]
    \item If \(\pi \delete \pi_1\) is decomposable and \(\pi_i > \pi_1\), then \(f(\pi)_{i+1} = \pi_i + 1\). \label{lem:f fixed 1}
    \item If \(\pi \delete \pi_1\) and \(\pi \delete 1\) are indecomposable, \(\pi \delete \pi_n\) is decomposable, and \(\pi_i < \pi_n\), then \(f(\pi)_i = \pi_i\). \label{lem:f fixed 2}
  \end{enumerate}
\end{lemma}

\begin{proof} \leavevmode
  For \ref{lem:f fixed 1} we have \((\pi \delete \pi_1)_{i-1} = \pi_i - 1\). The length of the first component of \(\pi \delete \pi_1\) is strictly less than \(\pi_1\), so \((f(\pi) \delete \pi_1)_i = \pi_i\) and \(f(\pi)_{i+1} = \pi_i + 1\).
    
  For part \ref{lem:f fixed 2}, \(f(\pi) = ({\rc} \circ f \circ {\rc})(\pi)\) by definition of \(f\), where \(\rc(\pi) \delete \rc(\pi)_1\) is decomposable. Since \(\rc(\pi)_1 = n + 1 - \pi_n\) and
  \begin{equation*}
    \rc(\pi)_{n + 1 - i} = n + 1 - \pi_{n + 1 - (n + 1 - i)} = n + 1 - \pi_i,
  \end{equation*}
  our assumption implies that \(\rc(\pi)_{n + 1 - i} > \rc(\pi)_1\). Part \ref{lem:f fixed 1} gives
  \begin{equation*}
    (f \circ {\rc})(\pi)_{n + 2 - i} = \rc(\pi)_{n + 1 - i} + 1 = n + 2 - \pi_i,
  \end{equation*}
  and therefore
  \begin{equation*}
    f(\pi)_i = ({\rc} \circ f \circ {\rc})(\pi)_i = n + 2 - (f \circ {\rc})(\pi)_{n + 2 - i} = \pi_i. \tag*{\qedhere}
  \end{equation*}
\end{proof}

\begin{lemma} \label{lem:not unremovable to removable}
  Let \(\pi \in \mathcal A_n^k\). If \(\pi \delete \pi_1\) is indecomposable, then \(f(\pi) \delete f(\pi)_1\) has at most two components.
\end{lemma}

\begin{proof}
  Denote \(\sigma = f(\pi)\) and \(\tau = \sigma \delete \sigma_1\). First, if \(\pi \delete 1\) is decomposable then \(\sigma \delete 1\) has at least three components and the claims follows from Proposition \ref{prop:forbidden removable point combinations}.

  Suppose instead that \(\pi \delete \pi_n\) is decomposable, and assume for the sake of a contradiction that \(\tau\) has at least three components. Denote the length of the first component of \(\tau\) by \(m\). Since the last \(n - \tau_n + 1\) entries of \(\tau\) must be contained in one component we have \(\tau_n > m + 1\), which implies that \(\sigma_{n+1} > m + 1\) and therefore that \(\pi_n > m\) by Remark \ref{rmk:f properties} \ref{rmk:f properties: edge values}. Lemma \ref{lem:f fixed points} \ref{lem:f fixed 1} shows that \(\sigma_i = \pi_i\) for all \(i\) for which \(\pi_i \leq m\), and we know that \(\sigma_2 \ldots \sigma_{m+1} = \tau_1 \ldots \tau_m\) is a permutation, so we get \(\pi_2 \ldots \pi_{m+1} = \tau_1 \ldots \tau_m\). Therefore \(\pi \delete \pi_1\) is decomposable, contradicting our original assumption.
  
  Lastly, suppose that \(\pi \delete n\) is decomposable, and assume again that \(\tau\) has at least three components, the first of which has length \(m\). In this case we have \(\tau^{-1}_n > m + 1\) which implies that \(\sigma^{-1}_{n+1} > m + 2\) (since an entry is added to the beginning) and therefore that \(\pi^{-1}_n > m + 1\) by Remark \ref{rmk:f properties} \ref{rmk:f properties: edge values}. Applying Lemma \ref{lem:f fixed points} \ref{lem:f fixed 2} to \(\pi^{-1}\) shows that \(\sigma^{-1}_i = \pi^{-1}_i\) for all \(i\) such that \(\pi^{-1}_i \leq m + 1\). Since \(\sigma_2 \ldots \sigma_{m+1}\) is a permutation, the condition \(\sigma^{-1}_i \leq m + 1\) holds for all \(i \in [m]\) and hence \(\pi^{-1}_1 \ldots \pi^{-1}_m = \sigma^{-1}_1 \ldots \sigma^{-1}_m\). We conclude that \(\pi_2 \ldots \pi_{m+1} = \tau_1 \ldots \tau_m\) is a permutation, and thus \(\pi \delete \pi_1\) is decomposable.
\end{proof}

\begin{theorem} \label{thm:injective}
  The function \(f: \mathcal A_n^k \to \Av_{n+1}^k(1324)\) is injective for all \(n\) and \(k\). Furthermore, \(\av_n^k(1324) \leq \av_{n+1}^k(1324)\) whenever \(k \leq 2n-7\).
\end{theorem}

\begin{proof}
  We will define a function \(h : f(\mathcal A_n^k) \to \mathcal A_n^k\) by reversing the steps used to define \(f\), and afterwards prove that \(h\) is the inverse of \(f\). Suppose that \(\sigma \in f(\mathcal A_n^k)\).
  \begin{itemize}
    \item If \(\sigma \delete \sigma_1\) has at least three components, let \(h(\sigma)\) be the permutation with \(h(\sigma)_1 = \sigma_1\) and \(h(\sigma) \delete \sigma_1 = g^{-1}(\sigma \delete \sigma_1)\).
    \item If \(\sigma \delete 1\) has at least three components, let \(h(\sigma) = h(\sigma^{-1})^{-1}\).
    \item Otherwise, let \(h(\sigma) = ({\rc} \circ h \circ {\rc})(\sigma)\).
  \end{itemize}
  Since at least one of the permutations \(\sigma \delete 1\), \(\sigma \delete \sigma_1\), \(\sigma \delete (n+1)\) or \(\sigma \delete \sigma_{n+1}\) has at least three components, we can always construct \(h(\sigma)\) according to the rules above. We will now prove that \(h\) is the inverse of \(f\). Let \(\pi \in \mathcal A_n^k\).
  \begin{itemize}
    \item If \(\pi \delete \pi_1\) is decomposable then \(f(\pi) \delete f(\pi)_1\) has at least three components by the definition of \(f\), and it is easy to see that \((h \circ f)(\pi) = \pi\).
    \item If \(\pi \delete 1\) is decomposable then \(f(\pi) \delete 1\) has at least three components, and \(f(\pi) \delete f(\pi)_1\) has at most two components by Lemma \ref{lem:not unremovable to removable}, so
    \begin{equation*}
      (h \circ f)(\pi) = h \big(f(\pi)^{-1}\big)^{-1} = h \big( f\big(\pi^{-1}\big) \big)^{-1} = \pi.
    \end{equation*}
    \item If \(\pi \delete 1\) and \(\pi \delete \pi_1\) are indecomposable, then both \(f(\pi) \delete 1\) and \(f(\pi) \delete f(\pi)_1\) have at most two components by Lemma \ref{lem:not unremovable to removable}. It follows that
    \begin{equation*}
      (h \circ f)(\pi) = ({\rc} \circ h \circ {\rc} \circ {\rc} \circ f \circ {\rc})(\pi) = \pi.
    \end{equation*}
  \end{itemize}

  We conclude that \(f\) has a left-inverse, and is therefore injective. All permutations in its image have at most two components by Remark \ref{rmk:f properties} \ref{rmk:f properties: two comps}, whereas the permutations in the image of \(g\) all have at least three components. Hence the images of \(f\) and \(g\) are disjoint, and the mapping
  \begin{equation*}
    \mathcal D_n^k \cup \mathcal A_n^k \longrightarrow \Av_{n+1}^k(1324)
  \end{equation*}
  given by combining \(f\) and \(g\) is injective. If \(k \leq 2n - 7\) then \(\mathcal D_n^k \cup \mathcal A_n^k = \Av_n^k(1324)\) by Theorem \ref{thm:removepoint}, so \(\av_n^k(1324) \leq \av_{n+1}^k(1324)\).
\end{proof}

\section{Enumerating the difference} \label{sec:difference}

The goal of this section is to describe the set of permutations 
\begin{equation*}
  \mathcal R_{n+1}^k \coloneqq \Av_{n+1}^k(1324) \setminus \big(g(\mathcal D_n^k) \cup f(\mathcal A_n^k)\big)
\end{equation*}
for all \(k \leq 2n - 7\). One obvious set of permutations in \(\mathcal R_{n+1}^k\) are all \(\sigma \in \Av_{n+1}^k(1324)\) with \(\sigma_1=n+1\) or \(\sigma_{n+1}=1\). To understand the remaining, first recall that \(g\) is a bijection from \(\mathcal D_n^k\) to permutations in \(\Av_{n+1}^k(1324)\) with at least three components. Key to the discussion is a natural extension of \(f^{-1}\). Let \(\mathcal B_{n+1}^k\) denote the set of permutations \(\sigma \in \Av_{n+1}^k(1324)\) such that \(\comp(\sigma) \leq 2\), \(\sigma_1 \neq n+1\), \(\sigma_{n+1} \neq 1\), and
\begin{equation} \label{eq:at least 3 comps condition}
  \max\{\comp(\sigma \delete i) : i = 1, n+1, \sigma_1, \sigma_{n+1} \} \geq 3.
\end{equation}
Then, as for \(f^{-1}\), we can define \(h: \mathcal B_{n+1}^k \to \mathcal A_n^k\) as follows:
\begin{itemize}
  \item If \(\comp(\sigma \delete \sigma_1) \geq 3\), let \(h(\sigma)\) be the permutation with \(h(\sigma)_1 = \sigma_1\) and \(h(\sigma) \delete \sigma_1 = g^{-1}(\sigma \delete \sigma_1)\).
  \item If \(\comp(\sigma \delete 1) \geq 3\), let \(h(\sigma) = h(\sigma^{-1})^{-1}\).
  \item Otherwise, let \(h(\sigma) = ({\rc} \circ h \circ {\rc})(\sigma)\).
\end{itemize}

Clearly \(h|_{f(\mathcal A_n^k)} = f^{-1}\). Furthermore, Proposition \ref{prop:2n-9 3 components} will show that whenever \(k \leq 2n-7\), \emph{all} permutations \(\sigma \in \Av_{n+1}^k(1324)\) with \(\comp(\sigma) \leq 2\) satisfy condition \eqref{eq:at least 3 comps condition}. Hence
\begin{equation*}
  f(\mathcal A_n^k) \cup \{\sigma \in \mathcal R_{n+1}^k : \sigma_1 \neq n+1 \text{ and } \sigma_{n+1} \neq 1\} = \mathcal B_{n+1}^k,
\end{equation*}
and we get
\begin{equation} \label{eq:remainder}
  \{\sigma \in \mathcal R_{n+1}^k : \sigma_1 \neq n+1 \text{ and } \sigma_{n+1} \neq 1\} = \{\sigma \in \mathcal B_{n+1}^k : f(h(\sigma)) \neq \sigma\}.
\end{equation}
Determining the permutations in the right-hand side of \eqref{eq:remainder} is a manageable task.

We will assume that \(k \leq 2n - 7\) throughout the remainder of this section. With the preceding discussion in mind, we reiterate that \(\mathcal R_{n+1}^k\) consists of the following collections.
\begin{enumerate}[R1]
  \item \label{diffclasses: corner} Permutations \(\sigma \in \Av_{n+1}^k(1324)\) with \(\sigma_1 = n+1\) or \(\sigma_{n+1} = 1\). In the first case
  \begin{equation*}
    \inv(\sigma \delete \sigma_1) = k - n \leq n - 7,
  \end{equation*}
  i.e.\ \(\sigma \delete \sigma_1\) is decomposable, and the other case is symmetrical, so these permutations are enumerated by
  \begin{equation*}
    [x^k]\big(2 x^n P(x)^2 \big),
  \end{equation*}
  where \(P(x) = \sum_{i \geq 0} p(i) x^i\) is the generating function for the partition numbers.
  \item \label{diffclasses: others} Permutations \(\sigma \in \Av_{n+1}^k(1324)\) with \(\comp(\sigma) \leq 2\), \(\sigma_1 \neq n+1\), \(\sigma_{n+1} \neq 1\), and \(f(h(\sigma)) \neq \sigma\). Again, \(h(\sigma)\) is always well-defined by Proposition \ref{prop:2n-9 3 components}.
\end{enumerate}
We will further split class \ref{diffclasses: others} into two subcollections to make its treatment easier:
\begin{enumerate}[R2a]
  \item \label{diffclasses: almost corner} Permutations \(\sigma \in \Av_{n+1}^k(1324)\) with \(\sigma_2 = n+1\) or \(\sigma_{n+1} = 2\). In the first case \(\comp(\sigma\backslash n+1)\ge 3\), so \(h(\sigma)_1 = n\), implying that \(h(\sigma) \delete h(\sigma)_1\) is decomposable and thus \(f(h(\sigma))_1 = n\). This means that \(f(h(\sigma)) \neq \sigma\), because if \(\sigma_1 = n\) then \(\sigma \delete (n+1)\) can not be decomposable. The other case is symmetrical. Similarly to \ref{diffclasses: corner}, this class is enumerated by
  \begin{equation*}
    [x^k] \big(2x^{n-1} P(x)^2\big).
  \end{equation*}
  \item \label{diffclasses: last} Permutations \(\sigma \in \Av_{n+1}^k(1324)\) with \(\comp(\sigma) \leq 2\) and \(\sigma_1, \sigma_2 \neq n+1\), \(\sigma_{n+1} \neq 1, 2\), such that \(f(h(\sigma)) \neq \sigma\). These permutations are described below, and their enumeration
  \begin{equation*}
    [x^k] \big(2x^{n-1} P(x)^2\big)
  \end{equation*}
  is obtained in Proposition \ref{prop:last part of difference}.
\end{enumerate}

We will start with an example of a class \ref{diffclasses: last} permutation.

\begin{center} \noindent
  \begin{tikzpicture}[scale=0.4]
    \node at (-5.3,5) {\(\sigma = 423167985 \quad = \)};
    \draw [step=1, thick, \gridcolor] (0.5,0.5) grid (9.5,9.5);
    \foreach \x\y in {1/4,2/2,3/3,4/1,5/6,6/7,7/9,8/8,9/5}
      \node at (\x,\y-0.004) {\pt};
  \end{tikzpicture}
\end{center}
Using the definitions of \(h\) (delete that last point) and \(f\) (delete the first point), we obtain the following permutations.

\begin{center}
  \begin{tikzpicture}[scale=0.4]
    \clip (0.5,-1.5) rectangle (8.5,8.5);
    \draw[step=1, thick, \gridcolor] (0.5,0.5) grid (8.5,8.5);

    \foreach \x\y in {1/5,2/2,3/3,4/1,5/6,6/8,7/7,8/4}
      \node at (\x,\y-0.004) {\pt};
    
    \node at (4.5,-1) {\(h(\sigma)\)};
  \end{tikzpicture}
  \hspace{5mm}
  \begin{tikzpicture}[scale=0.4]
    \clip (0.5,-1) rectangle (9.5,9.5);
    \draw[step=1, thick, \gridcolor] (0.5,0.5) grid (9.5,9.5);

    \foreach \x\y in {1/5,2/2,3/3,4/1,5/4,6/7,7/9,8/8,9/6}
      \node at (\x,\y-0.004) {\pt}; 
    
    \node at (4.5,-0.5) {\(f(h(\sigma))\)};
  \end{tikzpicture}
\end{center}
Clearly \(f(h(\sigma)) \neq \sigma\). Observe the structure of \(\sigma\): \(\sigma_1\) is placed precisely above the first component of \(\sigma \delete \{\sigma_1, \sigma_{n+1}\}\), and \(\sigma_{n+1} = \sigma_1 + 1\). This is not a coincidence, as it turns out that up to taking inverses, all class \ref{diffclasses: last} permutations have this structure.

\begin{lemma} \label{lem:r2b structure}
  If \(\sigma \in \Av_{n+1}^k(1324)\) is a class \ref{diffclasses: last} permutation, then
  \begin{enumerate}[(a)]
    \item \label{lem:r2b a} \(\comp(\sigma \delete \sigma_{n+1}) \geq 3\), \(\sigma_1 = m + 1\) and \(\sigma_{n+1} = m + 2\), where \(m\) is the length of the first component of \(\sigma \delete \sigma_{n+1}\), or
    \item \label{lem:r2b b} \(\comp(\sigma \delete (n+1)) \geq 3\), \(\sigma^{-1}_1 = m+ 1\) and \(\sigma^{-1}_{n+1} = m + 2\), where \(m\) is the length of the first component of \(\sigma \delete (n+1)\).
  \end{enumerate}
\end{lemma}

\begin{proof}
  Let \(\sigma \in \Av_{n+1}^k(1324)\) be a class \ref{diffclasses: last} permutation and denote \(\pi = h(\sigma)\). First observe that if \(\sigma \delete \sigma_1\) (resp.\ \(\sigma \delete 1\)) has at least three components, then \(\pi \delete \pi_1\) (resp.\ \(\pi \delete 1\)) is decomposable and \(f(\pi) = \sigma\), i.e.\ \(\sigma\) is not a class \ref{diffclasses: last} permutation -- this is easy to see from the definitions of \(f\) and \(h\). We will therefore assume throughout the proof that \(\sigma \delete \sigma_1\) and \(\sigma \delete 1\) have at most two components.
  
  Therefore \(\sigma \delete \sigma_{n+1}\) or \(\sigma \delete (n+1)\) has at least three components. Assume the former. Observe that if both \(\pi \delete \pi_1\) and \(\pi \delete 1\) are indecomposable then \(f(\pi) = \sigma\), so one of the two must be decomposable.
  \begin{itemize}
    \item Suppose first that \(\pi \delete \pi_1\) is decomposable. Consider the permutation \(\tau = \pi \delete \{\pi_1, \pi_n\}\) and write \(\tau = \tau^{(1)} \oplus 1 \oplus \ldots \oplus 1 \oplus \tau^{(2)}\) for its decomposition into components, \(m = |\tau^{(1)}|\) and \(\ell = |\tau^{(2)}|\). We must have \(\pi_1 > \pi_n\) in order for both \(\pi \delete \pi_1\) and \(\pi \delete \pi_n\) to be decomposable. Furthermore, if \(\pi_1 \leq m+1\) then \(\pi_n \leq m\) and \(\pi \delete \pi_1\) cannot be decomposable. Therefore \(\pi_1 \geq m+2\), and with a symmetrical argument \(\pi_n \leq n - \ell - 1\). Figure \ref{fig:pi structure} visualizes the structure of \(\pi\). Using the definition of \(h\), if \(b \geq 1\) (resp.\ \(b = 0\)) then \(\sigma\) is of the form described in Figure \ref{fig:sigma structure} (resp.\ Figure \ref{fig:sigma structure 2}), where \(b = \pi_1 - \pi_n - 1\).
    
    If \(b \geq 1\), the increasing sequence after \(\tau^{(1)}\) in \(\sigma\) is of length \(a+1 \geq 1\), which means that \(\sigma \delete \sigma_1\) has at least three components, contradicting our assumption. Therefore \(b = 0\). Similarly in this case, if \(a \geq 1\) then \(\sigma \delete \sigma_1\) has at least three components. Hence \(a = 0\), giving us condition \ref{lem:r2b a}.
    \item One can show that \(\pi \delete 1\) cannot be decomposable with a similar structural analysis.
  \end{itemize}
  If instead \(\comp(\sigma \delete (n+1)) \geq 3\), condition \ref{lem:r2b b} follows from a symmetrical argument, concluding the proof.
\end{proof}

\begin{figure}[p]
  \begin{minipage}[t]{0.5\textwidth}\vspace{0pt}
    \begin{tikzpicture}[scale=0.61]
      \draw[fill=blue!40] (1,1) rectangle (3,3);
      \draw[fill=blue!40] (6,7) rectangle (8,9);

      \fill[red!40] (0.5,4) rectangle (8.5,4.5);
      \fill[red!40] (0.5,5.5) rectangle (8.5,6);

      \draw (0.5,1) rectangle (8.5,9);
      \draw[very thick]
        (8.5,3) -- (8.5,4)
        (8.5,4.5) -- (8.5,5.5)
        (8.5,6) -- (8.5,7);
      \foreach \y in {3,4,4.5,5.5,6,7}
        \draw (0.5,\y) -- (8.6,\y);

      \foreach \x\y in {3.5/3.5,4.5/5,5.5/6.5}
        \foreach \d in {-0.1,0,0.1}
          \node[dot, minimum size=1pt] at (\x+\d,\y+\d) {};
      \foreach \x\y in {3.2/3.2,3.8/3.8,4.2/4.7,4.8/5.3,5.2/6.2,5.8/6.8}
        \node at (\x,\y-0.004) {\scriptsize\upt};
      \foreach \x\y in {0.75/5.75,8.25/4.25}
        \node at (\x,\y-0.004) {\pt};

      \node at (2,2) {\(\tau^{(1)}\)};
      \node at (7,8) {\(\tau^{(2)}\)};

      \node[anchor=west] at (8.7,3.5) {\(a\)};
      \node[anchor=west] at (8.7,5) {\(b\)};
      \node[anchor=west] at (8.7,6.5) {\(c\)};
      \node at (0.8,8.7) {\(\pi\)};
    \end{tikzpicture}
  \end{minipage}\hfill
  \begin{minipage}[t]{0.5\textwidth}\vspace{-0.5mm}
    \caption{A permutation \(\pi = h(\sigma)\) satisfying assumptions from the proof of Lemma \ref{lem:r2b structure}. We know that \(\pi_n\) lies above \(\tau^{(1)}\), \(\pi_1\) lies below \(\tau^{(2)}\) and \(\pi_1 > \pi_n\). The lengths of the increasing subsequences located (in height) from \(\tau^{(1)}\) to \(\pi_n\), \(\pi_n\) to \(\pi_1\) and \(\pi_1\) to \(\tau^{(2)}\) are denoted \(a\), \(b\) and \(c\), respectively.}
    \label{fig:pi structure}
  \end{minipage}
\end{figure}

\begin{figure}[p]
  \begin{minipage}[t]{0.5\textwidth}\vspace{0pt}
    \begin{tikzpicture}[scale=0.61]
      \draw[fill=blue!40] (1,1) rectangle (3,3);
      \draw[fill=blue!40] (6,7) rectangle (8,9);

      \fill[red!40] (0.5,4) rectangle (8.5,4.5);
      \fill[red!40] (0.5,5.5) rectangle (8.5,6);

      \draw (0.5,1) rectangle (8.5,9);
      \draw[very thick]
        (8.5,3) -- (8.5,4)
        (8.5,4.5) -- (8.5,5.5)
        (8.5,6) -- (8.5,7);
      \foreach \y in {3,4,4.5,5.5,6,7}
        \draw (0.5,\y) -- (8.6,\y);

      \foreach \x\y in {3.5/3.5,4.5/5,5.5/6.5}
        \foreach \d in {-0.1,0,0.1}
          \node[dot, minimum size=1pt] at (\x+\d,\y+\d) {};
      \foreach \x\y in {3.2/3.2,3.8/3.8,4.2/4.7,4.8/5.3,5.2/6.2,5.8/6.8}
        \node at (\x,\y-0.004) {\scriptsize\upt};
      \foreach \x\y in {0.75/5.75,8.25/4.25}
        \node at (\x,\y-0.004) {\pt};

      \node at (2,2) {\(\tau^{(1)}\)};
      \node at (7,8) {\(\tau^{(2)}\)};

      \node[anchor=west] at (8.7,3.5) {\(a+1\)};
      \node[anchor=west] at (8.7,5) {\(b-1\)};
      \node[anchor=west] at (8.7,6.5) {\(c+1\)};
      \node at (0.8,8.7) {\(\sigma\)};
    \end{tikzpicture}
  \end{minipage}\hfill
  \begin{minipage}[t]{0.5\textwidth}\vspace{-0.5mm}
    \caption{The permutation \(\sigma\) corresponding to \(\pi\) from Figure \ref{fig:pi structure}. If \(b \geq 1\), the lengths of the increasing sequences change from \(a\), \(b\) and \(c\) in \(\pi\) to \(a+1\), \(b-1\) and \(c+1\) in \(\sigma\), respectively.}
    \label{fig:sigma structure}
  \end{minipage}
\end{figure}

\begin{figure}[p]
  \begin{minipage}[t]{0.5\textwidth}\vspace{0pt}
    \begin{tikzpicture}[scale=0.61]
      \draw[fill=blue!40] (1,1) rectangle (3,3);
      \draw[fill=blue!40] (5,6) rectangle (7,8);

      \fill[red!40] (0.5,4) rectangle (7.5,5);

      \draw (0.5,1) rectangle (7.5,8);
      \draw[very thick]
        (7.5,3) -- (7.5,4)
        (7.5,5) -- (7.5,6);
      \foreach \y in {3,4,4.5,5,6}
        \draw (0.5,\y) -- (7.6,\y);

      \foreach \x\y in {3.5/3.5,4.5/5.5}
        \foreach \d in {-0.1,0,0.1}
          \node[dot, minimum size=1pt] at (\x+\d,\y+\d) {};
      \foreach \x\y in {3.2/3.2,3.8/3.8,4.2/5.2,4.8/5.8}
        \node at (\x,\y-0.004) {\scriptsize\upt};
      \foreach \x\y in {0.75/4.25,7.25/4.75}
        \node at (\x,\y-0.004) {\pt};

      \node at (2,2) {\(\tau^{(1)}\)};
      \node at (6,7) {\(\tau^{(2)}\)};

      \node[anchor=west] at (7.7,3.5) {\(a\)};
      \node[anchor=west] at (7.7,5.5) {\(c+1\)};
      \node at (0.8,7.7) {\(\sigma\)};
    \end{tikzpicture}
  \end{minipage}\hfill
  \begin{minipage}[t]{0.5\textwidth}\vspace{-0.5mm}
    \caption{The permutation \(\sigma\) corresponding to \(\pi\) from Figure \ref{fig:pi structure}, in the case \(b = 0\).}
    \label{fig:sigma structure 2}
  \end{minipage}
\end{figure}

\begin{proposition} \label{prop:last part of difference}
  If \(k \leq 2n-7\), the number of class \ref{diffclasses: last} permutations in \(\Av_{n+1}^k(1324)\) equals
  \begin{equation*}
    [x^k] \big( 2x^{n-1} P(x)^2 \big).
  \end{equation*}
\end{proposition}

\begin{proof}
  Let \(\sigma \in \Av_{n+1}^k(1324)\) be a class \ref{diffclasses: last} permutation satisfying Lemma \ref{lem:r2b structure} \ref{lem:r2b a} and let \(\tau = \sigma \delete \{\sigma_1, \sigma_{n+1}\}\). Then \(\inv(\tau) = k - n + 1 \leq n - 6\), so \(\tau \in \Av_{n-1}^{k-n+1}(1324)\) and
  \begin{equation*}
    \av_{n-1}^{k-n+1}(1324) = [x^{k-n+1}] P(x)^2 = [x^k] \big(x^{n-1} P(x)^2\big).
  \end{equation*}
  Conversely, for any permutation \(\tau \in \Av_{n-1}^{k-n+1}(1324)\), we can construct a unique class \ref{diffclasses: last} permutation \(\sigma \in \Av_{n+1}^k(1324)\) satisfying Lemma \ref{lem:r2b structure} \ref{lem:r2b a} by placing \(\sigma_1\) and \(\sigma_{n+1}\) appropriately. The case of Lemma \ref{lem:r2b structure} \ref{lem:r2b b} is symmetrical, and the result follows.
\end{proof}

We still need to show that \eqref{eq:at least 3 comps condition} is satisfied by all permutations in \(\Av_{n+1}^k(1324)\) whenever \(k \leq 2n-7\) in order to justify \eqref{eq:remainder}. This is the last part of the proof of Theorem~\ref{thm:1324}. We first show two lemmas we need.

\begin{lemma} \label{lem:132 decomposable}
  If \(\pi \in \Av_n^k(132)\) is indecomposable and \(k \leq 2n-5\), then \(\pi_1 = n\) and \(\pi \delete n\) is decomposable; or \(\pi_n = 1\) and \(\pi \delete 1\) is decomposable.
\end{lemma}

\begin{proof}
  Suppose that \(\pi_1 \neq n\) and \(\pi_n \neq 1\). Since \(\pi\) is indecomposable, \(\pi_n \neq n\). It follows that \(\pi_1 > \pi_n\), since \(\pi_1 n \pi_n\) forms a \(132\)-pattern otherwise. Similarly \(\pi^{-1}_1 > \pi^{-1}_n\) to avoid the \(132\)-pattern \(1 n \pi_n\). Counting inversions like in Section \ref{sec:almost decomposable} gives
  \begin{equation*}
      \inv(\pi) \geq \pi_1 - 1 + n - \pi_n + \pi^{-1}_1 - 1 + n - \pi^{-1}_n - 2 \geq 2n - 4.
  \end{equation*}
  If \(\pi_1 = n\), then \(\inv(\pi \delete n) \leq n-6\) and 
  and \(\pi \delete n\) must be decomposable. The other case is symmetrical.
\end{proof}

\begin{lemma} \label{lem:2n-9 2 to 5 components}
  If \(\sigma \in \Av_m^k(1324)\) is indecomposable, \(k \leq 2m-9\), and \(\sigma \delete \sigma_1\) has exactly two components, then \(\sigma \delete \sigma_m\) or \(\sigma \delete m\) has at least five components.
\end{lemma}

\begin{proof}
  Write \(\sigma \delete \sigma_1 = \sigma^{(1)} \oplus \sigma^{(2)}\) and let \(\delta = \sigma_1 - |\sigma^{(1)}|\). Since \(\sigma^{(1)}\) is indecomposable we have \(\inv \big(\sigma^{(1)}\big) \geq |\sigma^{(1)}| - 1\), and therefore
  \begin{align*}
    \inv \big( \sigma^{(2)} \big) &= k - (\sigma_1 - 1) - \inv \big( \sigma^{(1)}\big) \\
    &\leq 2m-9 - |\sigma^{(1)}| - \delta + 1 - |\sigma^{(1)}| + 1 \\
    &= 2\big(m - |\sigma^{(1)}|\big) - 7 - \delta \\
    &= 2|\sigma^{(2)}| - 5 - \delta. 
  \end{align*}
  Since \(\sigma^{(2)}\) is \(213\)-avoiding, \(\sigma^{(2)}_1 = |\sigma^{(2)}|\) or \(\sigma^{(2)}_{|\sigma^{(2)}|} = 1\) by Lemma \ref{lem:132 decomposable}, using the fact \(213 = \rc(132)\). In the first case
  \begin{align*}
    \comp \big(\sigma^{(2)} \delete \sigma^{(2)}_1\big) &\geq |\sigma^{(2)} \delete \sigma^{(2)}_1| - \inv\big(\sigma^{(2)} \delete \sigma^{(2)}_1\big) \\
    &\geq |\sigma^{(2)}| - 1 - \big( 2|\sigma^{(2)}| - 5 - \delta - |\sigma^{(2)}| + 1\big) = \delta + 3.
  \end{align*}
  Denoting \(\tau = \sigma \delete \sigma_1\), this shows that \(\tau \delete (m-1                                                            )\) has at least \(\delta + 4\) components, the first of which is \(\sigma^{(1)}\). Since \(\sigma_1 = |\sigma^{(1)}| + \delta\), `placing back' \(\sigma_1\) in front of \(\tau \delete (m-1)\) combines \(\delta\) components into one, so we conclude that \(\sigma \delete m\) has at least five components. The case \(\sigma^{(2)}_{|\sigma^{(2)}|} = 1\) is similar.
\end{proof}

\begin{proposition} \label{prop:2n-9 3 components}
  If \(\sigma \in \Av_m^k(1324)\) has at most two components and \(k \leq 2m-9\), then
  \begin{equation*}
      \max\{\comp(\sigma \delete i) : i = 1, m, \sigma_1, \sigma_m \} \geq 3.
  \end{equation*}
\end{proposition}

\begin{proof}
  If \(\sigma = \sigma^{(1)} \oplus \sigma^{(2)}\), then \(\inv \big(\sigma^{(i)} \big) \leq 2n - 5\) for \(i = 1\) or \(i = 2\), and Lemma \ref{lem:132 decomposable} proves the claim. If \(\sigma\) is indecomposable, then it is almost decomposable, i.e.\ one of the entries \(i \in \{1,m,\sigma_1,\sigma_m\}\) satisfies \(\comp(\sigma \delete i) \geq 2\). If \(\sigma \delete i\) has exactly two components, then (taking the inverse and reverse-complement if necessary) Lemma \ref{lem:2n-9 2 to 5 components} proves that another entry \(j \in \{1,m,\sigma_1,\sigma_m\}\) satisfies the claim.
\end{proof}

\begin{proof}[Proof of Theorem \ref{thm:1324}]
  Summing together the enumerations for classes \ref{diffclasses: corner}, \ref{diffclasses: almost corner} and \ref{diffclasses: last} yields
  \begin{equation*}
    |\mathcal R_{n+1}^k| = \av_{n+1}^k(1324) - \av_n^k(1324) = [x^k] \big(2(2 + x)x^{n-1} P(x)^2 \big) \eqqcolon [x^k] R_n(x)
  \end{equation*}
  whenever \(k \leq 2n-7\). It follows that
  \begin{align*}
    \av_n^k(1324) &= [x^k]\big(P(x)^2\big) - \sum_{i \leq k} [x^i] R_n(x) \\
    &= [x^k] \left(P(x)^2 - \frac{R_n(x)}{1-x} \right),
  \end{align*}
  and extracting the coefficient yields the result.
\end{proof}

\section{Further directions and conjectures} \label{sec:discussion}

This section contains discussion of three separate topics: extending our method; repeated differences of the numbers \(\av_n^k(1324)\); and the unimodality of the sequences
\begin{equation*}
  \av_n^0(1324), \ \ \av_n^1(1324), \ \ \ldots, \ \ \av_n^{\binom n 2}(1324).
\end{equation*}
Unimodality could improve the upper bound for \(L(1324)\) that Conjecture \ref{conj:1324anders} gives.

\subsection{Extending almost-decomposability}

A natural idea is to delete more than one point from the boundary of the plot of a permutation to make it decomposable. For example, one could imagine handling our earlier counterexample \(\pi = 3612745\) by deleting entries \(1\) and \(2\). However, the critical Proposition \ref{prop:forbidden removable point combinations} does not appear to have a counterpart, since e.g.\ \(\pi = 3412\) satisfies both
\begin{equation*}
    \comp(\pi \delete \{1,2\}) = 2 \qquad \text{and} \qquad
    \comp(\pi \delete \{\pi_1,\pi_2\}) = 2
\end{equation*}
Furthermore, some permutations with as few as \(2n-5\) have no natural point combination to remove. Consider the permutation
\begin{equation*}
  \pi = 34 \ldots (n-4) \ \ (n-2) 1 n 2 (n-1) (n-3).
\end{equation*}
Here is the first such permutation.
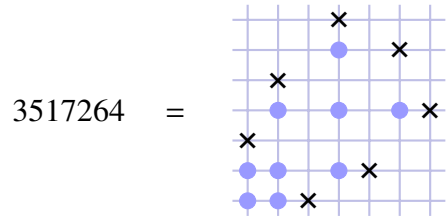
\begin{center}
  \begin{tikzpicture}[scale=0.4]
    \draw[step=1, thick, \gridcolor] (0.5,0.5) grid (7.5,7.5);
    \foreach \x\y in {1/3,2/5,3/1,4/7,5/2,6/6,7/4}
      \node at (\x,\y-0.004) {\pt};
    \foreach \x\y in {1/1,1/2,2/1,2/2,2/4,4/2,4/4,4/6,6/4}
      \node[dot, blue!40, minimum size=6pt] at (\x,\y) {};
    \node at (-3.9,4) {\(3517264 \quad = \)};
  \end{tikzpicture}
\end{center}
It is difficult to see which entries should be deleted. The permutation \(\pi \delete \{1,2\}\) is certainly decomposable, but applying \(g\) to it and inserting back the entries \(1\) and \(2\) yields a permutation containing \(1324\). One could instead try to delete different pairs of points, such as \(n\) and \(\pi_n\), but it is hard to imagine that the resulting mapping could be injective.

It should, still, be possible to improve the upper bound from \(k \leq 2n - 7\) to \(k \leq 2n - 6\) by analyzing the permutations in \(\Av_n^{2n-6}(1324)\) that are neither decomposable nor almost decomposable. These permutations are essentially characterized in the proof of  Lemma \ref{lem:no corner and both sides}; all other intermediate results in Section \ref{sec:almost decomposable} yield higher bounds for \(k\). There are not many such permutations, and they can probably be counted.

\subsection{Repeated differences}

One way to prove Conjecture \ref{conj:1324anders} would be understanding the numbers \(\av^n_k(1324)\) also for \(k>2n-7\).  Studying the numbers when \(k\le 3n-15\) we have found a tantalising pattern. We here extend the study of differences from Theorem \ref{thm:1324} to repeated differences. In what follows we will write \(\av_n^k\) for \(\av_n^k(1324)\). For \(n\ge 10\) it seems for example to always be the case that
\begin{equation*}
  (\av_{n+3}^{2n-3}-\av_{n+2}^{2n-3})-(\av_{n+2}^{2n-4}-\av_{n+1}^{2n-4})-\big((\av_{n+2}^{2n-5}-\av_{n+1}^{2n-5})-(\av_{n+1}^{2n-6}-\av_{n}^{2n-6})\big)=4.
\end{equation*}
In fact, for a fixed \(r\ge 0\), the following always seems to be constant for \(n\ge 10+r\):
\begin{align*}
  b_{r,n} &\coloneqq (\av_{n+3}^{2n+r-3}-\av_{n+2}^{2n+r-3})-(\av_{n+2}^{2n+r-4}-\av_{n+1}^{2n+r-4}) \\
  &\quad -\big((\av_{n+2}^{2n+r-5}-\av_{n+1}^{2n+r-5})-(\av_{n+1}^{2n+r-6}-\av_{n}^{2n+r-6})\big).
\end{align*}
 
Anders Claesson has kindly provided us with data of \(\av_n^k\) for \(k,n\le 45\) and with these we may compute numbers \(b_{r,n}\) up to \(r=9\) as 
\begin{equation*}
  4,8,14,28,52,88,150,244,390,612,
\end{equation*}
and we offer the following conjecture.

\begin{conjecture}\label{conj:3diff} 
  The numbers \(b_{r,n}\) are equal for a fixed \(r\) with \(n\ge 10+r\) (call them \(b_r\)) and they satisfy
  \begin{equation*}
      \sum_{r \geq 0} b_rx^r = 2\left(1-x^2\right)\left(2-x^2\right) P(x)^2,
  \end{equation*}
  where \(P(x)\) is again the generating function for the partition numbers.
\end{conjecture}
 
\begin{remark} 
  To guess the formula in the conjecture one really only needs five numbers \(b_0,b_1,b_2,b_3,b_4\) since the polynomial is of degree four, but it is true for all ten numbers we have.
\end{remark}
 
\begin{remark} 
  If the conjecture is proven we could still not determine the numbers \(\av_n^k\) for all \(k\le 3n-15\) since we also would need starting values when \(n=10+r\) of \((\av_{n+2}^{2n+r-5}-\av_{n+1}^{2n+r-5})-(\av_{n+1}^{2n+r-6}-\av_{n}^{2n+r-6})\). That sequence starts 
  \begin{equation*}
    12, 24, 41, 120, 274, 553, 1098, 2055
  \end{equation*}
  but we have no conjecture for what they are in general.
\end{remark}
 
To add an extra level of intrigue there seems to be a pattern also for fourth differences when trying to understand how the third differences \(b_{r,n}\) deviate for \(n<10+r\). More precisely, \(b_{r,r+9}\) and \(b_{r,r+8}\) seem to be \(4(r-2)\) and \(16(r-5)+32\) larger than in Conjecture \ref{conj:3diff} for \(r\ge 3\) and 6, respectively.

\subsection{Improved bounds and unimodality}

Let \(c \leq 1\) be a constant such that the maximal value of each sequence \(\big(\av_n^k(1324)\big)_k\) occurs with \(k \leq c \cdot \binom n2\), and assume that Conjecture \ref{conj:1324anders} is true. Denote \(m_n = \binom n2\), \(c_n = \left\lfloor c \cdot m_n\right\rfloor\) and \(\rho = \exp\big(\pi \sqrt{2/3}\big)\). Then, using the same technique as in \cite[Theorem 17]{claesson_upper_2012},
\begin{align*}
  \av_n(1324) = \sum_k \av_n^k(1324) &\leq (m_n + 1) \max_k \av_n^k(1324) \\
  &\leq (m_n + 1) [x^{c_n}] P(x)^2 \\
  &\leq (m_n + 1) (c_n + 1) \rho^{\sqrt{2 c_n}} \\
  &\leq (m_n + 1) (c_n + 1) \rho^{\sqrt c n \sqrt{1 - 1/n}}.
\end{align*}
Taking the \(n\)th root and the limit as \(n \to \infty\),
\begin{equation*}
  L(1324) \leq \rho^{\sqrt c} = \exp \big(\pi \sqrt{2c / 3} \big).
\end{equation*}
Incidentally, \(c = 21/23 \approx 0.913\) gives \(L(1324) \leq 11.6004\), which is in the range of the estimation \(L(1324) = 11.600 \pm 0.003\) from \cite{conway_1324-avoiding_2015,conway_1324-avoiding_2018}. We do not believe that an improvement this drastic will be possible with this approach, since the sequences \(\big(\av_n^k(1324)\big)_k\) should be very top-heavy. It is possible to produce large random 1324-avoiders using the Monte Carlo method in \cite{madras_random_2010}, and they seem to support this intuition.

The method of estimation should, furthermore, be too rough to get an upper bound very close to the actual value of \(L(1324)\). Note that \(c > 0.813\), since \(c = 0.813\) gives \(L(1324) \leq 10.263\), contradicting the known lower bound \(10.27\) -- again, assuming Conjecture \ref{conj:1324anders}.

\begin{question}
    Is there a constant \(c < 1\) such that the maximal value of each sequence \(\big(\av_n^k(1324)\big)_k\) occurs with \(k \leq c \cdot \binom n2\)?
\end{question}

This line of thinking leads to another natural question: are the sequences \((\av_n^k(1324)\big)_k\) unimodal? It is well-known (see \cite{bona_combinatorial_2004} for a nice proof) that \(\big(s_n^k\big)_k\) is log-concave, where \(s_n^k\) denotes the number of all permutations (not required to avoid any pattern) of length \(n\) with \(k\) inversions. As far as we know, there are no similar nontrivial results for the pattern avoiding case.

Log-concavity does not hold for \((\av_n^k(1324)\big)_k\), since e.g.\ \(2^2 < 1 \cdot 5\). On the other hand, if we remove the first \(n-1\) entries of each sequence, they do seem to be log-concave. Unimodality holds for the full sequences in the data we have, but we have not found a proof. Of special interest would be the position of the `tops' of the unimodal sequences, due to the discussion above.

\begin{conjecture}
  The sequence \(\left(\av_n^k(1324)\right)_{k=0}^{\binom n2}\) is unimodal for each \(n\).
\end{conjecture}

\subsection*{Acknowledgements}

We are grateful to Anders Claesson and Bjarki \'Ag\'ust Gu\dh{}mundsson, who provided us with data for the numbers \(\av_n^k(1324),\) for \( k,n\le 45\), and correctly guessed the generating function \(R_n(x) = 2(2 + x)x^{n-1} P(x)^2\) for the differences \(\av_{n+1}^k(1324) - \av_n^k(1324)\). The data is available at \url{https://akc.is/inv-mono/}. We are also grateful to the anonymous referee for their careful reading and helpful suggestions. Both authors were funded by the Swedish Research Council, VR, grant 2022-03875.

\printbibliography

\end{document}